\documentclass[11pt]{amsart}
\usepackage[utf8]{inputenc}
\usepackage{bm}
\usepackage{fontenc}
\usepackage{amsfonts}
\usepackage{amssymb}
\usepackage{amsmath}
\usepackage{amsthm}\usepackage{mathtools}
\usepackage{enumerate}
\usepackage{enumitem}
\usepackage[pagebackref,colorlinks,linkcolor=blue,citecolor=blue,urlcolor=blue,hypertexnames=true]{hyperref}
\usepackage{mathrsfs}
\usepackage{tikz}
\usetikzlibrary{calc}
\usepackage{marginnote}
\usepackage{xcolor}
\usepackage{soul}
\usepackage[all,cmtip]{xy} 
\usepackage{caption}

\newcommand{\R}{\mathbb{R}}
\newcommand{\Z}{\mathbb{Z}}

\newcommand{\N}{\mathbb{N}}
\newcommand{\C}{\mathbb{C}}
\newcommand{\B}{\mathcal{B}}
\newcommand{\cB}{\mathcal{B}}

\newcommand{\cE}{\mathcal{E}}

\newcommand{\cF}{\mathcal{F}}

\newcommand{\eps}{\varepsilon}

\newcommand{\SOTh}{\mathrm{SOT}\text{-}}

\newcommand{\Lip}{\mathrm{Lip}}
\newcommand{\cstu}{\mathrm{C}^*_u}
\newcommand{\cstql}{\mathrm{C}^*_{\textit{ql}}}

\newcommand{\cstar}{$\mathrm{C}^*$}

\newcommand{\CP}{\mathrm{CP}}
\newcommand{\AP}{\mathrm{AP}}

\numberwithin{equation}{section}

\newtheorem{theorem}{Theorem}[section]
\newtheorem*{theorem*}{Theorem}
\newtheorem{theoremi}{Theorem}
\newtheorem{proposition}[theorem]{Proposition}

\newtheorem*{proposition*}{Proposition}
\newtheorem{lemma}[theorem]{Lemma}
\newtheorem*{lemma*}{Lemma}
\newtheorem{corollary}[theorem]{Corollary}
\newtheorem*{corollary*}{Corollary}

\newtheorem*{fact*}{Fact}
\theoremstyle{definition}
\newtheorem{definition}[theorem]{Definition}
\newtheorem*{definition*}{Definition}

\newtheorem*{acknowledgments}{Acknowledgments}
\newtheorem*{AIusage}{AI usage statement}
\newtheorem*{leanformalization}{Lean formalization}
\newtheorem{claim}[theorem]{Claim}
\newtheorem*{claim*}{Claim}

\newtheorem*{conjecture*}{Conjecture}

\newtheorem{assumption}[theorem]{Assumption}

\theoremstyle{remark}
\newtheorem{example}[theorem]{Example}
\newtheorem*{example*}{Example}
\newtheorem{remark}[theorem]{Remark}
\newtheorem*{remark*}{Remark}
\newtheorem*{question*}{Question}

\DeclareMathOperator{\dom}{dom}
\DeclareMathOperator{\supp}{supp}

\DeclareMathOperator{\propg}{prop}

\DeclareMathOperator{\rank}{rank}
\DeclareMathOperator{\ad}{ad}

\begin{document}

\title[Dynamical \cstar-algebras and coarse geometry]{Dynamical \cstar-algebras and coarse geometry}%

\date{\today} 

\author[B. M. Braga]{Bruno M. Braga}
\address[B. M. Braga]{IMPA, Estrada Dona Castorina 110, 22460-320, Rio de Janeiro, Brazil}
\email{demendoncabraga@gmail.com}
\urladdr{https://sites.google.com/site/demendoncabraga}
\thanks {B. M. Braga  was partially supported by FAPERJ, grant E-26/204.317/2025,  by CNPq, grant 303182/2026-1, and by Serrapilheira, grant R-2501-51476.}

\begin{abstract}
Both  the uniform Roe algebras and the quasi-local algebras   encode the large scale geometry of metric spaces. Whether these algebras  coincide is a question which goes back to Roe and has only been recently solved by Ozawa. We propose a dynamical point of view on this problem. Given a set $X$, every map $h\colon X\to\mathbb{R}$ induces a one-parameter group $\sigma_h$ of automorphisms of $\mathcal{B}(\ell_2(X))$, given by conjugation by the diagonal unitaries $e^{ith}$, and we show that the operators which are continuity points of $\sigma_h$ are precisely the uniform Roe algebra of the pseudo-metric induced by $h$. If we moreover assume that $X$ is a uniformly locally finite metric space and $h$ is allowed to range over all coarse maps, this gives dynamical characterizations of both algebras: $\cstql(X)$ consists exactly of the operators which are continuous for all such flows, while $\cstu(X)$ consists exactly of norm limits of those which are analytic. The regularity conditions lying between continuity and analyticity then give rise to   \cstar-algebras between $\cstu(X)$ and $\cstql(X)$ which are invariant under bijective coarse equivalence, and we show that this scale is not degenerate: if $X$ is a coarse disjoint union of expander graphs, the algebra generated by operators which are analytic on a strip for every diagonal flow sits strictly between $\cstu(X)$ and $\cstql(X)$.
\end{abstract}
\maketitle


\section{Introduction}

Uniform Roe algebras are $\mathrm{C}^*$-algebras which encode the large scale
geometry of metric spaces. Introduced by Roe in the context of index theory
of elliptic operators on noncompact manifolds (see \cite{Roe1988,Roe1993}), they have since
become a central object at the interface of coarse geometry, operator algebras,
and higher index theory.  Given a  metric space
$(X,d)$, $\ell_2(X)$ denotes the Hilbert space of complex-valued functions on $X$ whose squares are absolutely summable and $(\delta_x)_{x\in X}$ denotes its canonical orthonormal basis. An operator $a\in\mathcal{B}(\ell_2(X))$ is said to have
\emph{finite propagation} if there is $r>0$ such that
$\langle a\delta_y,\delta_x\rangle=0$ whenever $d(x,y)>r$, and the
\emph{uniform Roe algebra} of $X$  is the norm
closure of the operators of finite propagation; we denote it by     $\cstu(X)$ or by $\cstu(X,d)$ in case we want to emphasize the metric. Deciding whether an operator belongs to  $\cstu(X)$ is, however, a very complicated task in general; to corroborate this, notice that membership in $\cstu(X)$  is defined by
an existential quantifier ranging over the uncountable set of finite propagation operators, and
this quantifier cannot be replaced by one ranging over a countable set (whenever $X$ is infinite).  This led Roe to consider a more flexible
notion \cite{Roe1988}: an operator $a\in \mathcal{B}(\ell_2(X))$ is
\emph{quasi-local} if for every $\varepsilon>0$ there is $r>0$ such that
$\|\chi_A a\chi_B\|\le \varepsilon$ whenever $A,B\subseteq X$ satisfy
$d(A,B)\ge r$ --- here, for $C\subseteq X$,  $\chi_C$  denotes the orthogonal projection on $\ell_2(X)$ with range $\ell_2(C)$. Quasi-local operators form a $\mathrm{C}^*$-algebra called  the
\emph{quasi-local algebra} of $X$; we denote it by $\mathrm{C}^*_{ql}(X)$ or by $\mathrm{C}^*_{ql}(X,d)$ in case some emphasis is needed. It is straightforward that  $\cstql(X)$ contains
$\mathrm{C}^*_u(X)$. Whether these two algebras coincide is a question which
goes back to Roe and which has attracted considerable attention in recent years.
On the positive side, $\mathrm{C}^*_u(X)=\mathrm{C}^*_{ql}(X)$ whenever $X$ has
Yu's property A (see \cite[Theorem 3.3]{SpakulaZhang2020JFA} for the original proof of this result or \cite[Theorem 5]{Ozawa2023} for a simpler proof). On the negative
side, Ozawa has recently shown that if $X$ is a coarse disjoint union of
expander graphs, then the product $\prod_n \mathrm{M}_n$ of matrix algebras embeds
into $\mathrm{C}^*_{ql}(X)$, while it does not embed into $\mathrm{C}^*_u(X)$
(see \cite[Theorems A and B]{Ozawa2023}). In particular, $\mathrm{C}^*_u(X)\subsetneq \mathrm{C}^*_{ql}(X)$
for such spaces.

Ozawa's result shows that, in general, there is room between
$\mathrm{C}^*_u(X)$ and $\mathrm{C}^*_{ql}(X)$. The goal of this paper is to
understand what lies in this gap. We do so by introducing a new, dynamical,
point of view on both algebras: we show that $\mathrm{C}^*_u(X)$ and
$\mathrm{C}^*_{ql}(X)$ are precisely the norm limits of operators which are, respectively,
\emph{analytic} and \emph{continuous} with respect to a canonical family of
one-parameter automorphism groups of $\mathcal{B}(\ell_2(X))$. This
characterization does more than reinterpret known objects: the whole scale of
regularity conditions lying between continuity and analyticity gives rise to
$\mathrm{C}^*$-algebras lying between $\mathrm{C}^*_u(X)$ and
$\mathrm{C}^*_{ql}(X)$, and we show that this scale is not degenerate. In
particular, we obtain a bijective coarse invariant $\mathrm{C}^*$-algebra which sits
strictly between the uniform Roe algebra and the quasi-local algebra of
expanders.

\subsection{The dynamical viewpoint}
Given a set $X$ and any map $h\colon X\to\mathbb{R}$, we let $\sigma_h=(\sigma_{h,t})_{t\in\mathbb{R}}$
be the one-parameter group of automorphisms of $\mathcal{B}(\ell_2(X))$ given by
conjugation by the diagonal unitaries $e^{ith}$. Precisely, under the canonical identification of $h$ with a  (possibly unbounded) self-adjoint multiplication operator on $\ell_2(X)$, $\sigma_h$ is given by
\[
\sigma_{h,t}(a)=e^{ith}ae^{-ith}\quad\text{for all } a\in\mathcal{B}(\ell_2(X))\text{ and } t\in\mathbb{R}.
\]
 We refer to $\sigma_h$ as the \emph{diagonal pre-flow} induced by $h$. Although
$\sigma_h$ is always a homomorphism $\mathbb{R}\to\mathrm{Aut}(\mathcal{B}(\ell_2(X)))$,
it is far from being a \emph{flow} in the usual sense, i.e., the orbit maps
$t\mapsto\sigma_{h,t}(a)$ are in general not norm continuous. This is why we
call it a \emph{pre-flow}.  Our starting point is \cite[Proposition 2.1]{BragaExel2023}: in the context of studying  KMS states on uniform Roe algebras, the authors showed that
$\sigma_h$ restricts to a genuine flow on $\mathrm{C}^*_u(X)$ precisely when
$h$ is a \emph{coarse map}, i.e., when \[\sup\{|h(x)-h(y)|\mid d(x,y)\le r\}<\infty\]
for all $r>0$. This suggests that the family of diagonal pre-flows indexed by
coarse maps should be able to detect the large scale geometry of $X$. The present paper confirms this in a strong sense. We shall now start describing our findings in detail.

Given a pre-flow $\sigma$ on a $\mathrm{C}^*$-algebra $A$, we denote by
$A^\sigma$ the $\mathrm{C}^*$-subalgebra of $A$ consisting of its
\emph{continuity points}, i.e.,
\[
A^\sigma=\{a\in A\mid t\mapsto \sigma_t(a)\ \text{is norm continuous}\}.
\]
Our first result identifies the continuity points of a diagonal pre-flow. Note
that it requires no geometry on $X$ at all: $X$ is merely a set, and the map
$h\colon X\to \R$ induces a pseudo-metric on it by letting  \[d_h(x,y)=|h(x)-h(y)|\] for all $x,y\in X$. For the next result, we consider the uniform Roe algebra given by this pseudo-metric, $\cstu(X,d_h)$; the definition is the same as the one above given by an actual metric.

\begin{theoremi}
Let $X$ be a set and $h\colon X\to\mathbb{R}$ be any function. Then
\[
\mathcal{B}(\ell_2(X))^{\sigma_h}=\mathrm{C}^*_u(X,d_h).
\]
More generally, if $A\subseteq\mathcal{B}(\ell_2(X))$ is a $\sigma_h$-invariant
$\mathrm{C}^*$-subalgebra, then $A^{\sigma_h}$ is the closure of the operators
in $A$ with finite $d_h$-propagation.\label{thmA}
\end{theoremi}

The main ingredients to prove Theorem \ref{thmA} are the
Fejér  and the   de la Vall\'ee Poussin kernels. In a nutshell, the Fejér kernel allows us to approximate in norm the  continuity points of $\sigma_h$ by elements of finite $d_h$-propagation. On the other hand, Fejér kernels together with  the   de la Vall\'ee Poussin kernel ensure that finite $d_h$-propagation secures continuity under $\sigma_h$. This result has been obtained by   Ewert and Meyer in the context of mathematical physics for the case where $X$ is a subspace of the grid $\Z^n$ (see \cite[Theorem 4]{EwertMeyer2019CMP}). 

Theorem \ref{thmA} is a statement about a single map $h$. Under metric assumptions on $X$, one can detect the geometry of $X$ by letting $h$ range over
all coarse maps. From now on, we shall focus on \emph{uniformly locally finite} (abbreviated as \emph{u.l.f.}) metric spaces: these are metric spaces such that for all $r>0$ there is an upper bound on the cardinality of the balls of radius $r$ in $X$ (an upper bound independent of the center of the ball). For instance, Cayley graphs of finitely generated groups have this property. 

Given a metric 
space $X$, let
\[
\mathrm{CP}(X)=\bigcap_{\substack{h\colon X\to\mathbb{R}\\ h\ \text{coarse}}}\mathcal{B}(\ell_2(X))^{\sigma_h}
\]
be the $\mathrm{C}^*$-algebra of operators which are continuity points of every
diagonal flow induced by a coarse map. On the other end of the regularity
spectrum, recall that an element $a$ of a $\mathrm{C}^*$-algebra $A$ equipped
with a flow $\sigma$ is \emph{analytic} if $t\mapsto\sigma_t(a)$ extends to an
entire function $\mathbb{C}\to A$. Analytic elements are a classical tool in the theory of
$\mathrm{C}^*$-dynamical systems: they are the elements on which the KMS
condition is formulated, and they form a dense subalgebra for every flow
(see \cite{BratteliRobinsonBookII1997} for the theory of KMS states and \cite{BragaExel2023} for KMS states in the context of uniform Roe algebras). Define: \begin{itemize}
    \item $\mathrm{AP}(X)$ is the norm closure of the operators which are
analytic for $\sigma_h$ for every coarse map $h\colon X\to \R$.
\end{itemize}
We refer to   Definitions \ref{Def.BandAnalytic.functions} and Theorem \ref{thm:roeentire} for details. and 
\ref{Def.AP.algebra.defi} for further details. Our second main result says that these
dynamically defined algebras are exactly the classical ones.

\begin{theoremi}\label{thmB}
Let $X$ be a u.l.f.\ metric space. Then
\[
\mathrm{C}^*_{ql}(X)=\mathrm{CP}(X)\quad\text{and}\quad \mathrm{C}^*_u(X)=\mathrm{AP}(X).\]
\end{theoremi}

In words: an operator is quasi-local if and only if it moves continuously under
every diagonal flow, and it belongs to the uniform Roe algebra if and only if it is a norm limit of operators that
move analytically  under every diagonal flow.  We show that an operator has finite propagation if and only if it is analytic for $\sigma_h$ for all coarse maps $h\colon X\to \R$ and, moreover, this analyticity is \emph{of exponential type}; we refer to Definitions \ref{Def.BandAnalytic.functions} and Theorem \ref{thm:roeentire} for details.  

Theorem \ref{thmB} places $\mathrm{C}^*_u(X)$ and $\mathrm{C}^*_{ql}(X)$ at the
 end points of a scale of regularity conditions, and one may now ask what the
intermediate regularity classes (Lipschitz orbits, differentiable orbits, analytic
orbits on a strip, etc.) correspond to. Each such class gives a
$\mathrm{C}^*$-algebra between $\mathrm{C}^*_u(X)$ and $\mathrm{C}^*_{ql}(X)$ which,
by construction, is a bijective coarse invariant of $X$.\footnote{Indeed, the definition
of these algebras only involves the coarse maps $X\to\mathbb{R}$, which depend
only on the coarse structure of $X$; see Section \ref{Section.Preliminaries} for the definition of a coarse structure.} When $X$ has property
A, all these algebras collapse to $\mathrm{C}^*_u(X)$. Our third main result shows
that for expanders, this is not the case. We concentrate on the natural weakening
of analyticity obtained by asking that $t\mapsto \sigma_{h,t}(a)$ extend
analytically only to a horizontal strip $\{z\in\mathbb{C}\mid |\mathrm{Im}(z)|\le\delta\}$ 
in the complex plane. Define: 
\begin{itemize}\item $\mathrm{AP}_{\mathrm{strip}}(X)$ is  the norm
closure of the operators which are analytic on a strip for $\sigma_h$ for every
coarse map $h\colon X\to \R$.\end{itemize}

\begin{theoremi}\label{thmC}
Let $X$ be a coarse disjoint union of expander graphs. Then
\[
\mathrm{C}^*_u(X)\subsetneq \mathrm{AP}_{\mathrm{strip}}(X)\subsetneq \mathrm{C}^*_{ql}(X).
\]
\end{theoremi}

To the best of our knowledge, $\mathrm{AP}_{\mathrm{strip}}(X)$ is the first
example of a  bijective coarse invariant $\mathrm{C}^*$-algebra lying strictly between the
uniform Roe algebra and the quasi-local algebra. 

\subsection{Quantifying quasi-locality}
The proof of Theorem \ref{thmC} goes through a different, purely metric, way
of navigating  between the uniform Roe and the quasi-local algebra, which is
of independent interest. Given $a\in\mathcal{B}(\ell_2(X))$, its
\emph{quasi-locality  modulus} is the function
\[
\varepsilon_a(r)=\sup\{\|\chi_A a\chi_B\|\mid A,B\subseteq X,\ d(A,B)\ge r\},\quad r\ge 0.
\]
Thus $a$ has finite propagation if and only if $\varepsilon_a$ is eventually
zero, and $a$ is quasi-local if and only if $\varepsilon_a(r)\to 0$ as $r\to\infty$.
Between these two extremes, one can prescribe the rate of decay: for
$\alpha>0$, we let \begin{itemize}\item $\mathrm{QL}_\alpha(X)$ be the norm closure of the operators
with $\varepsilon_a(r)=O(r^{-\alpha})$, and
\item $\mathrm{QL}_{\exp}(X)$ be the norm
closure of the operators whose modulus decays exponentially.
\end{itemize}
These are
$\mathrm{C}^*$-algebras and
\[
\mathrm{C}^*_u(X)\subseteq \mathrm{QL}_{\exp}(X)\subseteq \mathrm{QL}_\beta(X)\subseteq \mathrm{QL}_\alpha(X)\subseteq \mathrm{C}^*_{ql}(X)
\]
for all $0<\alpha<\beta$.  Our fourth main result shows that, for expanders,
this is a genuine continuum of distinct $\mathrm{C}^*$-algebras.

\begin{theoremi}\label{thmD}
Let $X$ be a coarse disjoint union of expander graphs. Then
\[
\mathrm{C}^*_u(X)\subsetneq \mathrm{QL}_{\exp}(X)\subsetneq \mathrm{QL}_\beta(X)\subsetneq \mathrm{QL}_\alpha(X)\subsetneq \mathrm{C}^*_{ql}(X)
\]
for all $0<\alpha<\beta$. Moreover, $\prod_n\mathrm{M}_n$ embeds into
$\mathrm{QL}_{\exp}(X)$.
\end{theoremi}

The last statement of Theorem \ref{thmD} strengthens Ozawa's embedding theorem
\cite[Theorem B]{Ozawa2023}: not only does $\prod_n\mathrm{M}_n$ embed into the
quasi-local algebra of an expander, but it does so as operators whose
quasi-locality modulus decays exponentially fast. The proofs of both
Theorems \ref{thmC} and \ref{thmD} rely on consequences of the concentration of measure
phenomenon on high dimensional spheres (see Lemma \ref{cor.lem:frame}), following the strategy of \cite{Ozawa2023,LiZhangZhu2026},
combined with a trace argument which allows us to detect the rate of decay of
quasi-locality moduli through the growth of powers of self-adjoint operators. 
 
We stress that, unlike the algebras $\mathrm{AP}_{\mathrm{strip}}(X)$, the
algebras $\mathrm{QL}_\alpha(X)$ and $\mathrm{QL}_{\exp}(X)$ are \emph{not} bijective 
coarse invariants: as the quasi-locality  modulus is not a bijective coarse invariant,  their definitions depend a priori on the metric, and not only on its
coarse equivalence class. However, for coarse disjoint unions of finite graphs, this ends up not being an issue and we obtain the following.

\begin{theoremi}\label{thmE}
Let $X$ be a u.l.f.\ metric space. Then
$\mathrm{AP}_{\mathrm{strip}}(X)\subseteq \mathrm{QL}_{\exp}(X)$. If moreover
$X$ is a coarse disjoint union of finite connected graphs, then
$\mathrm{AP}_{\mathrm{strip}}(X)= \mathrm{QL}_{\exp}(X)$.
\end{theoremi}

The first inclusion in Theorem \ref{thmE} is obtained by a Baire category argument on the compact space
of $1$-Lipschitz functions on $X$, which produces a single strip width working
uniformly for all $1$-Lipschitz maps; the converse inclusion for disjoint unions of
graphs is obtained through explicit estimates on the iterated commutators
$\mathrm{ad}_h^k(a)=[h,[h,\dots[h,a]\dots]]$. Theorem \ref{thmC} then follows by
combining Theorems \ref{thmB}, \ref{thmD} and \ref{thmE}. In particular,
Theorem \ref{thmE} shows that, at least for coarse disjoint unions of graphs, norm approximation by operators with 
exponential decay of their quasi-locality modulus, a metric condition, is in fact
a bijective  coarse invariant.

\subsection{Coarse spaces}
Several of our results hold, and are proved, for \emph{coarse} spaces. We postpone to Section \ref{Section.Preliminaries} the precise definition of such spaces. For now, we simply say  that these are pairs $(X,\cE)$, where $X$ is a set and $\cE$ a family of subsets of $X\times X$ satisfying some permanence properties, which serve as abstractions of metric spaces and code their large scale aspects.   In this setting, while an arbitrary map
$h\colon X\to\mathbb{R}$ need not be coarse, there is always a largest coarse substructure $\cE_h$ of $\cE$ for which $h$ is coarse 
and we prove that
\[
\mathrm{C}^*_u(X,\mathcal{E})^{\sigma_h}=\mathrm{C}^*_u(X,\mathcal{E}_h)
\quad\text{and}\quad
\mathrm{C}^*_{ql}(X,\mathcal{E})^{\sigma_h}=\mathrm{C}^*_{ql}(X,\mathcal{E}_h)
\]
(see Theorems \ref{Thm.uRa.h.Points.Cont.Substructure} and \ref{Thm.qla.h.Points.Cont.Substructure}). In particular, the inclusion
$\mathrm{C}^*_{ql}(X,\mathcal{E})\subseteq \mathrm{CP}(X,\mathcal{E})$ holds for
every coarse space; interestingly, the inclusion can be strict for
non-metrizable u.l.f.\ coarse structures (Remark \ref{RemarkContPointsCoarseSpaces}), so that the
equality $\mathrm{C}^*_{ql}(X)=\mathrm{CP}(X)$ in Theorem \ref{thmB} is a
genuinely metric phenomenon.

\section{Preliminaries}\label{Section.Preliminaries}

While we have introduced most of the main definitions needed for this paper in the introduction, a couple of them were left out. This   short section takes care of this.

 \subsection{Coarse geometry}\label{SubsectionCoarseSpaces}
Let us briefly introduce the notion of coarse spaces, coarse embeddings, and explain how metric spaces are viewed as coarse spaces. For a detailed treatment of the subject, we refer the reader to the excellent monograph \cite{RoeBook} --- for coarse geometry in the more restrictive context of metric spaces, see \cite{NowakYuBook}.

Coarse spaces are abstractions of metric spaces which aim to code large scale notions of the spaces such as boundedness properties  but ignore all small scale phenomena. For this, let $X$ be a set and $\cE $ be a collection of subsets of $X\times X$. We say that $\cE$ is a \emph{coarse structure} if
\begin{enumerate}
    \item $\Delta_X=\{(x,x)\in  X^2\mid x\in X\}\in \cE$,
    \item $F\in \cE$ for all $F\subseteq E$ where $E\in \cE$,
    \item $E\cup F\in \cE$ for all $E,F\in \cE$,
    \item $\{(y,x)\in X^2\mid (x,y)\in E\}\in \cE$ for all $E\in \cE$, and
    \item $\{(x,y)\in X^2\mid \exists z\in X,\ \text{ such that }\ (x,z)\in E\ \text{ and }\ (z,y)\in F\}\in \cE$ for all $E,F\in \cE$.
\end{enumerate}
In this case, we call the pair $(X,\cE)$ a \emph{coarse space}. The quintessential example of a coarse space is a metric space: if $(X,d)$ is a metric space, then
\[\cE_d=\left\{E\subseteq  X^2\mid \sup_{(x,y)\in E}d(x,y)<\infty\right\}\]
is a coarse structure on $X$. In general, a coarse structure $\cE$  on $X$ is \emph{metrizable} if there is a metric $d$ on $X$ such that $\cE=\cE_d$. Viewing a metric space as a coarse space helps to interpret coarse spaces better: for a given symmetric $E\in \cE$ containing $\Delta_X$,
\[E_x=\{y\in X\mid (x,y)\in E\}\]
should be interpreted as the ``coarse version'' of a ball (of finite radius) centered at $x$. We then define a   coarse space $(X,\cE)$ to be  \emph{uniformly locally finite} (abbreviated as \emph{u.l.f.}) if
\[\sup_{x\in X}|E_x|<\infty\]
for all symmetric $E\in \cE$.

In the category of coarse spaces, the morphisms are the \emph{coarse maps}: a map $f\colon (X,\cE)\to (Y,\cF)$ is \emph{coarse} if
\[(f\times f)[\cE]\subseteq \cF,\]
i.e., if for all $E\in \cE$ there is $F\in \cF$ such that
\[(x,y)\in E\ \text{ implies }\ (f(x),f(y))\in F.\]
If the target space is metric, say $(Y,d_Y)$,  this can be quantified as in the introduction: $f\colon (X,\cE)\to (Y,d_Y)$ is coarse if
\[\omega_f(E)=\sup\left\{d_Y(f(x),f(z))\mid  (x,z)\in E\right\}<\infty\]
for all $E\in \cE$. If the domain is also metric, say $(X,d_X)$, we simply  write  
\[\omega_f(r)=\sup\left\{d_Y(f(x),f(z))\mid d_X (x,z)\leq r \right\}\]
for all $r\geq 0$.  Throughout the paper, we will only consider coarse maps from a coarse space   into $\R$. The real line is always considered with its standard metric.

Finally, we say that coarse spaces $(X,\cE)$ and $(Y,\cF)$ are \emph{coarsely equivalent} if there are coarse maps $f\colon X\to Y$ and $g\colon Y\to X$ such that $g\circ f$ and $f\circ g$ are \emph{close} to the identity maps $\mathrm{Id}_X$ and $\mathrm{Id}_Y$, respectively, i.e., if
\[\{(x,g(f(x)))\in X^2\mid x\in X\}\in \cE\ \text{ and }\ \{(y,f(g(y)))\in Y^2\mid y\in Y\}\in \cF.\]
If, moreover, $f$ can be taken to be bijective, then $X$ and $Y$ are \emph{bijectively coarsely equivalent}.

\subsection{Uniform Roe and quasi-local algebras} While these were defined in the introduction for metric spaces, we now properly define them in the more general context of coarse spaces. 

Let $(X,\cE)$ be a coarse space. Every operator $a\in  \cB(\ell_2(X))$ is seen as an $X$-by-$X$ matrix and its \emph{support} is defined as 
\[\supp(a)=\{(x,y)\in X^2\mid \langle a\delta_y,\delta_x\rangle\neq 0\}.\]
The operator $a$ is said to have \emph{controlled propagation} if $\supp(a)\in \cE$ --- if we need to emphasize the coarse structure, we say that the operator $a$ is \emph{$\cE$-controlled}. In the case of a metric space $(X,d)$, we also say that $a$ has \emph{finite propagation} and write
\[\propg(a)=\sup\{d(x,y)\mid \langle a\delta_y,\delta_x\rangle\neq 0\}.\]

\begin{definition}
    Let $(X,\cE)$ be a coarse space. The \emph{uniform Roe algebra of $X$}, denoted by $\cstu(X,\cE)$, is the norm closure of the set of all operators in $\cB(\ell_2(X))$ with controlled propagation. We simply write  $\cstu(X)$ if $\cE$ is clear from the context.
    \end{definition}

Given a coarse space $(X,\cE)$, $\eps>0$ and $E\in \cE$, we say that $a$ is \emph{$(\eps,E)$-quasi-local} if 
for all $A,B\subseteq X$,
\[(A\times B)\cap E=\emptyset\  \text{ implies }\ \|\chi_A a\chi_B\|\leq \eps.\]
If for every $\eps>0$ there is $E\in \cE$ such that $a$ is $(\eps,E)$-quasi-local, then $a$ is called \emph{quasi-local}. Notice that, in the case of a metric space $(X,d)$, this can be translated in terms of \emph{$(\eps,r)$-quasi-locality} meaning that 
for all $A,B\subseteq X$,
\[d(A,B)\geq r \  \text{ implies }\ \|\chi_A a\chi_B\|\leq \eps.\]
\begin{definition}
    Let $(X,\cE)$ be a coarse space. The set of all quasi-local operators in $\cB(\ell_2(X))$ forms a \cstar-algebra called the   \emph{quasi-local algebra of $X$}, denoted by $\cstql(X,\cE)$. If $\cE$ is clear from the context, we simply write $\cstql(X)$.
\end{definition}

It is completely straightforward that, for any coarse space $(X,\cE)$, we have  $\cstu(X)\subseteq \cstql(X)$.

\subsection{Coarse disjoint unions and expander graphs}\label{Subsection.Coarse.Disj.Union}

Let $(X_n,d_n)_n$ be a sequence of finite metric spaces. A metric space $(X,d)$ is said to be a \emph{coarse disjoint union of $(X_n)_n$} if  $X=\bigsqcup_nX_n$ and the metric $d$ satisfies that
\begin{enumerate}
    \item\label{Item.Metric.p.cdu.1} $d\restriction X_n\times X_n=d_{n}$ for all $n\in\N$, and
    \item \label{Item.Metric.p.cdu.2} $d(X_n,X_m)\to \infty$ as $n+m\to \infty$ with $n\neq m$.
\end{enumerate}
We emphasize here that, while coarse disjoint unions are not uniquely defined from the metric point of view (this justifies the indefinite article used in its definition), they are uniquely defined coarsely. Precisely, any two coarse disjoint unions of $(X_n)_n$ are bijectively coarsely equivalent to each other.

Among coarse disjoint unions, expander graphs will be of utmost importance. Recall that, for $k\in \N$ and $\gamma>0$, a finite (undirected) graph $G=(V,E)$ is a \emph{$(k,\gamma)$-expander} if every vertex is incident to at most $k$ edges, and for all $A\subseteq V$ we have that
\[|A|\leq |V|/2 \text{ implies } |\partial A|\geq \gamma|A|,\]
where $\partial A=\{v\in V\setminus A\mid \exists u\in A, \ (v,u)\in E\}$. A $(k,\gamma)$-expander is automatically connected. Identifying it with its set of vertices, we view it  as a metric space  endowed with the shortest path metric. We say that a metric space $X$ is a \emph{coarse disjoint union of expander graphs} if there are $k\in\N$ and $\gamma>0$ such that $X$ is a coarse disjoint union of a sequence $(X_n)_n$ such that each $X_n$ is a $(k,\gamma)$-expander and $\lim_n|X_n|=\infty$. By this definition, as $k$ is taken to be independent of $n$, coarse disjoint unions of expander graphs are automatically  uniformly locally finite. For further details on expander graphs, we refer to the monograph \cite{LubotzkyBook2010}.

We single out now a property of expander graphs which will be essential for us. In fact, this will be the only property of a  coarse disjoint union of expander graphs which we will actually need for our theorems. Precisely, given a coarse disjoint union $X=\bigsqcup_nX_n$ of $(k,\gamma)$-expanders $(X_n)_n$, it is straightforward that there is $\kappa>1$ such that 
for all $n\in\N$ and all $A,B\subseteq X_n$ we have 
\begin{equation*}
\min\left\{\frac{|A|}{|X_n|},\frac{|B|}{|X_n|}\right\}\leq \kappa^{-\frac{d(A,B)}{2}}.
\end{equation*}
(see  \cite{LiNowakSpakulaZhang2021GGD} for  \emph{asymptotic expanders}).

\subsection{Weak integral}

 Throughout the paper, we will repeatedly integrate functions taking values in some $\cB(H)$. The integral used will be the \emph{weak integral} and we  quickly recall its definition here. Suppose $(\Sigma,\mu)$ is a measure space, $H$ is a Hilbert space,  and $f\colon \Sigma\to \cB(H)$ is measurable. Then, for each $\xi,\zeta\in H$, the map
 \[\sigma\in\Sigma\mapsto \langle f(\sigma)\xi,\zeta\rangle\in \C\]
 is also measurable and, assuming furthermore that this map is integrable,
 \[(\xi,\zeta)\in H\times H\mapsto \int_\Sigma\langle f(\sigma)\xi,\zeta\rangle d\mu(\sigma) \in \C\]
 is a sesquilinear form. If, moreover, this sesquilinear form is bounded, then there is a unique operator, say $T\in \cB(H)$, such that
 \[\langle T\xi,\zeta\rangle=\int_\Sigma\langle f(\sigma)\xi,\zeta\rangle d\mu(\sigma) \]
 for all $\xi,\zeta$. The \emph{weak integral of $f$} is defined as
 \[\int_\Sigma f d\mu =T.\]

\section{Algebra of continuity points of a diagonal flow}\label{SectionAlgContPoint}

In this section, we study diagonal  flows given by arbitrary maps $h\colon X\to \R$, where $X$ is simply a set. No a priori  geometric assumptions are made on $X$. Instead,  each  map $h$ induces a pseudo-metric
\begin{align*}
    d_h\colon  X\times X &\to [0,\infty)\\
    (x,y) &\mapsto |h(x)-h(y)|
\end{align*}
and we show that the \cstar-algebra of continuity points of $\sigma_h$ is precisely the uniform Roe algebra $\cstu(X,d_h)$ --- recall, as mentioned in the introduction, the uniform Roe algebra with respect to a pseudo-metric is defined in the exact same way. The main result proved in this section is Theorem \ref{thmA}.

 Before proving Theorem \ref{thmA}, we introduce some notation which will be used not only here but also in Section \ref{SectionSubalgebrasofURAQLContPOint}; this follows ideals in \cite[Section 2.2]{EwertMeyer2019CMP}.   Let $X$ be a set and $h\colon X\to \R$ be a map. By Stone's theorem, the one-parameter group of unitaries $t\in \R\mapsto e^{ith}\in \cB(\ell_2(X))$ is strongly continuous in the sense that
\[t\in \R\mapsto e^{ith}\xi\in \ell_2(X)\]
is continuous for all $\xi\in \ell_2(X)$. As a consequence, the map
\[t\in \R\mapsto \langle \sigma_{h,t}(a)\xi,\zeta\rangle\in \C\] is continuous for all $\xi,\zeta\in\ell_2(X)$.
Therefore, for each $f\in L_1(\R)$ and each $a\in \cB(\ell_2(X))$, the sesquilinear form
\[(\xi,\zeta)\in \ell_2(X)\times \ell_2(X)\mapsto \int_\R f(t)\langle \sigma_{h,t}(a)\xi,\zeta\rangle dt\]
is well-defined and bounded; here we are also using that $ \sigma_{h,t}$ is norm preserving. We define    $\Theta_{h,f}(a)$ to be the weak integral \begin{equation}\label{Eq.Theta.formula}\Theta_{h,f}(a)=\int_\R f(t) \sigma_{h,t}(a)  d\lambda(t),
\end{equation}
where $\lambda$ denotes the Lebesgue measure on $\R$.
 Clearly,
\begin{equation}\label{Eq.BoundNormThetafa}\|\Theta_{h,f}(a)\|\leq \|f\|_1\|a\|.
\end{equation}

For each $f\in L_1(\R)$ and each $a\in \cB(\ell_2(X))$, the operator $\Theta_{h,f}$ acts on $a$ as a Schur multiplier. Precisely, letting
\[\hat f(\theta)=\int_\R f(t)e^{it\theta}d\lambda(t)\ \text{ for all }\ \theta\in \R,\]
a straightforward computation gives that
\begin{align}\label{Eq.Theta.Schur.mult}
\langle \Theta_{h,f}(a)\delta_y,\delta_x\rangle& = \int_\R f(t)\langle e^{ith}ae^{-ith}\delta_y,\delta_x\rangle d\lambda(t)  \\
&= \int_\R f(t) e^{it(h(x)-h(y))} d\lambda(t) \cdot\langle a\delta_y,\delta_x\rangle\notag \\
&=\hat f(h(x)-h(y))\cdot \langle a\delta_y,\delta_x\rangle\notag
\end{align}
for all $x,y\in X$.
So, $\Theta_{h,f}$ is simply the Schur multiplication by the $X$-by-$X$ matrix \[[\hat f(h(x)-h(y))]_{x,y\in X}.\]

  For each $s>0$, let $K_s\colon \R\to [0,\infty)$ be the Fejér kernel for $\R$, i.e.,
\begin{equation}\label{Eq.Fejér.kernel}
K_s(t)=\frac{s}{2\pi}\left(\frac{\sin(s t/2)}{s t /2}\right)^2
\end{equation}
for all $t\in \R$ (we refer to the monograph \cite{Katznelson2004} for more on Fejér kernels). Direct computations give that
\begin{equation}\label{Eq.2.03.Aug.26}\|K_s\|_1=1\ \text{ and }\ \hat K_s(\theta)=\max\left\{0,1-\frac{|\theta|}{s}\right\} \end{equation}
for all $s>0$ and all $\theta\in \R$.

\begin{lemma}
\label{Lemma.support.ThetaKsa}
Let $X$ be a set, $h\colon X\to \R$ be a map, and $a\in \cB(\ell_2(X))$. Then, for each $s>0$, $\Theta_{h,K_s}(a)$ has $d_h$-propagation at most $s$.
\end{lemma}

\begin{proof}
Given $s>0$, we know that $\Theta_{h,K_s}$ is the Schur multiplication by the $X$-by-$X$ matrix $[\hat{K_s}(h(x)-h(y))]_{x,y}$ and, by the formula for $\hat K_s$ in \eqref{Eq.2.03.Aug.26}, we have that $\hat K_s(\theta)=0$ if $|\theta|\geq s$. Therefore, if $x,y\in X$ are such that \[d_h(x,y)=|h(x)-h(y)|>s,\] the entry $\langle \Theta_{h,K_s}(a)\delta_y,\delta_x\rangle$ must be zero as desired.
\end{proof}

\begin{lemma}\label{Lemma.ThetaKs(a).tends.to.a}
    Let $X$ be a set and $h\colon X\to \R$ be a map. If $a\in \cB(\ell_2(X))^{\sigma_h}$, then \[a=\lim_{s\to \infty}\Theta_{h,K_s}(a).\]
\end{lemma}

\begin{proof}
As $K_s$ is positive and $\|K_s\|_1=1$  (see \eqref{Eq.2.03.Aug.26}), we have that
\begin{align}\label{Eq.1.03.Aug.26}\|\Theta_{h,K_s}(a)-a\|
\leq \int_\R K_s(t)\|\sigma_{h,t}(a)-a\|dt
\end{align}
for all $s>0$.

Let us show that \eqref{Eq.1.03.Aug.26} approaches zero when $s\to \infty$. Fix $\eps>0$. As $t\in \R\mapsto \sigma_{h,t}(a)\in \B(\ell_2(X))$ is continuous at zero, there is $\delta>0$ such that
\[\|\sigma_{h,t}(a)-a\|\leq \eps\ \text{ for all }\ t\in [-\delta,\delta].\]
Using again that $K_s$ has integral 1,   the formula of $K_s$, and the fact that  $\|\sigma_{h,t}(a)-a\|\leq 2\|a\|$  give
\begin{align*}
    \int_\R K_s(t) & \|\sigma_{h,t}(a)-a\|dt \\
    &=\int_{|t|\leq \delta} K_s(t)\|\sigma_{h,t}(a)-a\|dt+\int_{|t|>\delta} K_s(t)\|\sigma_{h,t}(a)-a\|dt
\\
&\leq \eps+\frac{4\|a\|}{s \pi}\int_{|t|>\delta}\frac{1}{t^2}dt  \\
&=\eps+\frac{8\|a\|}{s \pi \delta}.
\end{align*}
Hence, if $s$ is at least ${8\|a\|}/({\eps\pi \delta})$, this shows that \eqref{Eq.1.03.Aug.26} is at most $2\eps$. As $\eps$ was arbitrary,  we conclude that
\[\|\Theta_{h,K_s}(a)-a\|\to 0\ \text{ as } \ s\to \infty\]
    as desired.
\end{proof}

\begin{proof}  [Proof of Theorem \ref{thmA}]
Let $h\colon X\to \R$ and suppose $A\subseteq \cB(\ell_2(X))$ is a \cstar-algebra invariant under $\sigma_h$. Let us show that 
\begin{equation}\label{eq:spectral}
A^{\sigma_h}=\overline{\bigcup_{s>0}\mathrm C^s[X,d_h]\cap A},
\end{equation}
where here for each $s>0$ we consider 
\begin{equation}\label{Eq.21.sep.26.1}\mathrm C^s[X,d_h]=\{a\in \cB(\ell_2(X))\mid d_h\text{-}\propg(a)\leq s\}.\end{equation}

Fix $a\in A^{\sigma_h}$. Let  $K_s$ be the  Fejér kernel (see \eqref{Eq.Fejér.kernel}) and   $\Theta_{h,K_s}$ be as above (see \eqref{Eq.Theta.formula}). By Lemma \ref{Lemma.support.ThetaKsa} each $\Theta_{h,K_s}(a)$ has $d_h$-propagation at most $s$ and, by Lemma \ref{Lemma.ThetaKs(a).tends.to.a}, we have that
\[a=\lim_{s\to \infty}\Theta_{h,K_s}(a).\]
Therefore, in order to show that $a$ is in the right-hand side of \eqref{eq:spectral}, we only need to show that each $\Theta_{h,K_s}(a)$ is in $A$. For that, notice that,   as $a$ is in $A^{\sigma_h}$, the map
\[t\in \R\mapsto \sigma_{h,t}(a)\in A\] is continuous. Therefore,
\[t\in \R\mapsto K_s(t)\sigma_{h,t}(a)\in A\]
 is Bochner integrable, which implies that  $\Theta_{h,K_s}(a)$ is not only a weak integral, but actually a Bochner integral in this case. As $A$ is norm closed, it follows that each $\Theta_{h,K_s}(a)$ is in $A$ as desired.

We now prove the inclusion ``$\supseteq$'' in \eqref{eq:spectral}. Since $A^{\sigma_h}$ is closed,  it suffices to show that    it   contains all operators in $A$ with finite $d_h$-propagation. For that, fix $s>0$ and let $a\in A $ be an operator with propagation at most $s$. Define \[V_s=2K_{2s}-K_s,\] i.e., $V_s$ is the  de la Vall\'ee Poussin kernel. Standard computations give that
\begin{equation}\label{Eq.5.03.Aug.26}\|V_s\|_1\leq 3\ \text{ and }\   \hat V_s(t)= 1 \text{ for all }\ t\in [-s,s]\end{equation}
(see \cite{Katznelson2004} for more on this kernel). Notice that since  $a$ has $d_h$-propagation at most $s$,   \eqref{Eq.Theta.Schur.mult} and the second equality in \eqref{Eq.5.03.Aug.26} imply that   $\Theta_{h,V_s}(a)=a$.

For each $r\in \R$, consider the translation
\[V^r_s(t)=V_s(t-r)\ \text{ for all }\ t\in \R.\]
So,
\[\|V_s^r-V_s\|_1\to 0\ \text{ as } \ r\to 0.\]
Notice that, as $\sigma_{h}$ is a one-parameter group, a simple change of variables gives that
\[\sigma_{h,r}\circ \Theta_{h,V_s}=\Theta_{h,V_s^r}.\]
Hence, using  \eqref{Eq.BoundNormThetafa} and the fact that  $\Theta_{h,V_s}(a)=a$, we have
\begin{align*}
\|\sigma_{h,r}(a)-a\| &=\|\sigma_{h,r}(\Theta_{h,V_s}(a))-\Theta_{h,V_s}(a)\|\\
&=\|\Theta_{h,V_s^r}(a)-\Theta_{h,V_s}(a)\|\\
&=\|\Theta_{h,V_s^r-V_s}(a)\|\\
&\leq\|{V_s^r}-{V_s}\|_1\|a\|
\\
&\underset{r\to 0}{\longrightarrow} 0.
\end{align*}
This shows that   $a$ is a continuity point of $\sigma_h$ and, therefore, it belongs to  $A^{\sigma_h}$ as desired.
\end{proof}

\section{Subalgebras of the uniform Roe  algebra and   of the quasi-local  algebra of continuity points of a diagonal flow}\label{SectionSubalgebrasofURAQLContPOint}

In this section,  we use the methods of Section \ref{SectionAlgContPoint} in order to characterize the subalgebras of  $\cstu(X)$ and $\cstql(X)$ given by the continuity points of diagonal flows. For this to make sense, $X$ will no longer be only a set but, moreover, a coarse space endowed with a coarse structure $\cE$. Our main results here show that  these algebras of continuity points are precisely the uniform Roe and the quasi-local algebras of the largest substructure of $\cE$ making $h$  coarse (see Theorems \ref{Thm.uRa.h.Points.Cont.Substructure} and \ref{Thm.qla.h.Points.Cont.Substructure}).

\subsection{Subalgebra of the uniform Roe algebra  of continuity points of a diagonal flow} For arbitrary maps $h\colon X\to \R$ on a coarse space $(X,\cE)$, we can define the largest coarse substructure of $\cE$ with respect to which $h$ is coarse. Precisely:

\begin{definition}\label{Definition.Eh.largest.coarse.substructure}
Let $(X,\cE)$ be a coarse space and $h\colon X\to \R$ be a map. Define 
\[\cE_h=\left\{E\in \cE\mid \sup_{(x,y)\in E}|h(x)-h(y)|<\infty\right\}.\]
\end{definition}
The following proposition is   completely  straightforward and  we omit its proof.

 \begin{proposition}\label{PropLargestCoarseStrucWithhCoarse}
      Let $(X,\cE)$  be  a coarse space and $h\colon X\to\R$ be a function. Then $(X,\cE_h)$ is the largest  coarse substructure of $\cE$ such that $h\colon (X,\cE_h)\to \R$ is coarse. \qed
 \end{proposition}

\begin{theorem}\label{Thm.uRa.h.Points.Cont.Substructure}
    Let $(X,\cE)$ be a  coarse space and $h\colon X\to \R$ be a map. Then
    \[\cstu(X,\cE)^{\sigma_h}=\cstu(X,\cE_h).\]
\end{theorem}

\begin{proof}
    We start with the inclusion ``$\supseteq$''. Let $a\in \cstu(X,\cE_{h})$. Then, by the definition of $\cE_h$, it is immediate that  $a\in \cstu(X,d_h)$. By Theorem \ref{thmA}, $a$ is a continuity point for $\sigma_{h}$. Therefore,   as  $\cE_h\subseteq \cE$, we have that $\cstu(X,\cE_h)\subseteq \cstu(X,\cE)$ and  this shows that     $a\in\cstu(X,\cE)^{\sigma_h} $.

For the inclusion ``$\subseteq$'', suppose now that $a\in \cstu(X,\cE)^{\sigma_h}$. By  Lemma \ref{Lemma.ThetaKs(a).tends.to.a},
\[a=\lim_{s\to \infty}\Theta_{h,K_s}(a).\]
So, we only need to show that each $\Theta_{h,K_s}(a)$ is in $ \cstu(X,\cE_h)$. Fix $s>0$ and let  $(a_n)_n$ be a sequence of operators of $\cE$-controlled propagation converging to $a$. So,
\[\Theta_{h,K_s}(a)=\lim_n\Theta_{h,K_s}(a_n)\]
and we are left to notice that each $\Theta_{h,K_s}(a_n)$ is in $\cstu(X,\cE_h)$.   Fix $n\in\N$. By Lemma \ref{Lemma.support.ThetaKsa}, each $\Theta_{h,K_s}(a_n)$ has $d_h$-propagation at most $s$. Since $\Theta_{h,K_s}$ acts on $a_n$ by a Schur multiplication, it cannot increase support, so  we have
\[\supp\left(\Theta_{h,K_s}(a_n)\right)\subseteq \supp(a_n)\cap \{(x,y)\in X^2\mid d_h(x,y)\leq s\}\in \cE_h.\]
This shows that each $\Theta_{h,K_s}(a_n)$ has controlled propagation with respect to $\cE_h$ and we are done.
\end{proof}

\subsection{Subalgebra of the  quasi-local algebra of  continuity points of a diagonal flow} We will now need Theorem \ref{thmA} in order to prove the following result.

\begin{theorem}\label{Thm.qla.h.Points.Cont.Substructure}
    Let $(X,\cE)$ be a coarse space and $h\colon X\to \R$ be a map. Then
    \[\cstql(X,\cE)^{\sigma_h}=\cstql(X,\cE_h).\]
\end{theorem}

The following lemma is
 \cite[Lemma 6]{Ozawa2023} and it first appeared in the proof of \cite[Theorem 2.8]{SpakulaTikuisis2019}. While in both these articles the spaces under consideration were metrizable, their proof works for non-metrizable spaces without modifications.

\begin{lemma}\emph{(}\cite[Lemma 6]{Ozawa2023}\emph{).}
Let $(X,\cE)$ be a  coarse space, $\eps,\delta>0$, $E\in \cE$,   $a\in\B(\ell_2(X))$ be $(\eps,E)$-quasi-local, and let $f\in\ell_\infty(X)$ be so that   $0\le f\leq 1$ and $|f(x)-f(y)|\le\delta$ for all $x,y\in X$ with $(x,y)\in E$. Then
\begin{equation*}
\|[f,a]\|\le 4\delta\|a\|+2\delta^{-1}\varepsilon .
\end{equation*} \label{lem:commutator}
\end{lemma}

\begin{lemma}\label{lem:smoothing}
Let $(X,\cE)$ be a coarse space, $h\colon X\to\R$ be a coarse map, $\eps>0$, $E\in \cE$,   $a\in\B(\ell_2(X))$  be  $(\eps,E)$-quasi-local, and $\delta\in (0,1]$.  Then for all $s>2\pi\omega_h(E)/\delta$, there is  $b\in \cB(\ell_2(X))$ with $d_h$-propagation at most $s$  such that
\begin{equation}\label{Eq.3.Aug.26.Depois.lhovav}
\|a-b\|\le 16\delta\|a\|+8\delta^{-1}\varepsilon .
\end{equation}
\end{lemma}

\begin{proof}
Fix    $\omega\in(\omega_h(E),s\delta/(2\pi))$ and, to simplify notation, let   $L=\tfrac{\pi\omega}{\delta}$. We start by fixing   a special partition of unity for $\R$ consisting of functions with supports centered around each $kL$, $k\in\Z$, and whose square roots are $\tfrac{\pi}{2L}$-Lipschitz. This can be done as follows: for each $k\in \Z$, let
\begin{equation}\varphi_k(t)=\left\{\begin{array}{ll}
\cos\left(\frac{\pi(t-kL)}{2L}\right) ,    & \ \text{ if }\ |t-kL|\leq L,\\
    0, \ \text{ otherwise}.&
\end{array}\right.
\end{equation}
It is clear that each $\varphi_k$ is $\tfrac{\pi}{2L}$-Lipschitz with support contained in $[(k-1)L,(k+1)L]$. Also, for each   $t\in \R$ there is $\ell\in\Z$ such that at most $\varphi_\ell$ and $\varphi_{\ell+1}$ are nonzero at $t$. Finally, notice that  $(\varphi_k^2)_{k\in\Z}$ forms a partition of unity of $\R$. Indeed, as $\cos(\alpha-\pi/2)=\sin(\alpha)$, let $t\in \R$ and $\ell\in \Z$ be as above and fix $r\in [0,1)$ such that $t=\ell L+rL$. Then,
\begin{align*}
\sum_{k\in\Z}\varphi_k^2(t) &=\varphi_\ell^2(t)+   \varphi_{\ell+1}^2(t)\\
&=\cos^2\left(\frac{\pi r}{2}\right)+\cos^2\left(\frac{\pi(r-1)}{2}\right)\\
&=\cos^2\left(\frac{\pi r}{2}\right)+\sin^2\left(\frac{\pi r}{2}\right)\\
&=1.
\end{align*}

For each $k\in \Z$, let $g_k=\varphi_k\circ h$ and view each $g_k$ as an operator in $\ell_\infty(X)$. So each $g_k$ is a positive contraction. Moreover, as $(\varphi_k^2)_{k\in\Z}$ is a partition of unity, it follows that
\[\text{SOT-}\sum_{k\in \Z}g_k^2=1_{\ell_2(X)}.\]
As each $\varphi_k$ is $\tfrac{\pi}{2L}$-Lipschitz, we also have that

    \begin{equation}\label{eq:lip}
|g_k(x)-g_k(y)|\le\frac{\pi}{2L}|h(x)-h(y)|
\end{equation}
for all $x,y\in X$.

For this $s>2L$, let us show that there is $b\in\cB(\ell_2(X))$ with $d_h$-propagation at most $s$  satisfying \eqref{Eq.3.Aug.26.Depois.lhovav}. Precisely, let
\[
b=\text{SOT-}\sum_{k\in\Z}g_kag_k
\]
and let us show it  has the desired properties. Firstly, notice that $b$ is   well-defined. Indeed, splitting the sum above into its indices in   $2\Z$ and   $2\Z+1$,  we obtain two sums of orthogonal operators with uniformly bounded norms and therefore they converge in the strong operator topology.

Let us now notice that $b$  has $d_h$-propagation at most $s$. Indeed, since $s>2L$, if $x,y\in X$ are so that $|h(x)-h(y)|=d_h(x,y)> s$, then $h(x)$ and $h(y)$ are not in the support of any single $\varphi_k$.  Therefore, for each $k\in\Z$, either $g_k(x)$ or $g_k(y)$ is always zero, which implies that
\[\langle b\delta_y,\delta_x\rangle=\sum_{k\in\Z} g_k(x)\langle a\delta_y,\delta_x\rangle g_{k}(y)=0 .\]
This shows that $b$  has $d_h$-propagation at most $s$.

We are left to estimate the norm $\|a-b\|$. For this, consider $\Sigma=\{-1,1\}^{\Z}$ endowed with its  product probability measure $\mu$, where each copy of $\{-1,1\}$ is considered with its uniform probability measure.  Let $(r_k)_{k\in \Z}$ be the Rademacher random variables, i.e., for each $k\in\Z$ and $\sigma=(\sigma_k)_{k\in\Z}\in \Sigma$, we have  $r_k(\sigma)=\sigma_k$. For each $\sigma\in \Sigma$, let \[f_{\sigma}=\SOTh\sum_kr_k(\sigma)g_k\in\ell_\infty(X).\] Since each $\varphi_k$ is $\tfrac{\pi}{2L}$-Lipschitz and since for each $x\in X$ there are at most two of these functions which do not vanish at $h(x)$, we have that
\begin{equation}\label{Eq.LipConstant.fsigma}
|f_{\sigma}(x)-f_{\sigma}(y)|\le\frac{2\pi}{L}|h(x)-h(y)|
\end{equation}
for all $x,y\in X$.

As the Rademacher variables are orthonormal, we have that
\[\int_{\Sigma}f_\sigma af_\sigma d\mu(\sigma)=\SOTh\sum_kg_kag_k=b\]
and
\[\int_{\Sigma}f_\sigma^2d\mu(\sigma)=\SOTh\sum_kg_k^2=1,\]
where the integrals above are all weak integrals. Hence,
\[
b-a=\int_\Sigma f_{\sigma}[a,f_{\sigma}]d\mu(\sigma) ,
\]
and by the operator-valued Cauchy--Schwarz inequality,
\begin{align}\label{eq:CS}
\|a-b\|& \leq\left\|\int f_{\sigma}^2 d\mu(\sigma)\right\|^{1/2}\left\|\int[a,f_{\sigma}]^*[a,f_{\sigma}] d\mu(\sigma)\right\|^{1/2}\notag\\
&\le\sup_{\sigma\in\Sigma}\|[a,f_{\sigma}]\|.
\end{align}

We are left to  estimate the last term in \eqref{eq:CS}. Fix $\sigma\in \Sigma$ and let
\[\tilde f=\frac{f_{\sigma}+2}{4}.\]
So,
\begin{equation}\label{Eq.o que eu fiz}
[\tilde f,a]=\frac{1}{4}[f_\sigma,a].
\end{equation}
As $\|f_\sigma\|_\infty\leq 2$,  it is clear that  $0\le\tilde f\le 1$. If $x,y\in X$ are so that $(x,y)\in E$, then  $|h(x)-h(y)|\le\omega_h(E)<\omega$ and, by \eqref{Eq.LipConstant.fsigma}, we have
\[
|\tilde f(x)-\tilde f(y)|\le \frac{2\pi\omega}{4L}=\frac{\delta}{2} .
\]
As $a$ is assumed to be $(\eps,E)$-quasi-local, Lemma \ref{lem:commutator} applied to $\tilde f$ implies that  \[\|[\tilde f,a]\|\le4\delta\|a\|+2\delta^{-1}\varepsilon\] and, using \eqref{Eq.o que eu fiz}, we conclude that
\begin{align*}
\|[f_{\sigma},a]\|&\le 16\delta\|a\|+8\delta^{-1}\varepsilon.
\end{align*}
Combining with \eqref{eq:CS} finishes the proof.
\end{proof}

\begin{theorem}\label{Thm.Quasi-LocalContainedContinuityPoints}
    Let $(X,\cE)$ be a   coarse space. Then $\cstql(X,\cE)\subseteq \CP(X,\cE)$.
\end{theorem}

\begin{proof}
    Let $a\in \cstql(X,\cE)$ be a contraction. As $\CP(X,\cE)$ is the intersection of all $\cB(\ell_2(X))^{\sigma_h}$ with $h$ ranging over   all coarse maps $  X\to \R$, we need to show that $a$ belongs to each of these sets. Fix such $h$. By Theorem \ref{thmA}, we only need to show that for any given $\gamma>0$ there are $s>0$ and $b\in \cB(\ell_2(X)) $ with $d_h$-propagation at most $s$ such that $\|a-b\|\leq \gamma$. Fix such $\gamma$ and let $\eps\in(0,1]$ be small enough so that $24\sqrt{\eps}\leq \gamma$. As $a$ is quasi-local, there is $E\in \cE$ such that $a$ is $(\eps,E)$-quasi-local. As $h$ is coarse, $\omega_h(E)$ is finite and we can pick $s>2\pi\omega_h(E)/\sqrt{\eps}$. Applying Lemma \ref{lem:smoothing} with  $\delta=\sqrt{\eps}$, we obtain $b\in \cB(\ell_2(X))$ with $d_h$-propagation at most $s$ such that
    \[\|a-b\|\leq 16\sqrt{\eps}+8\sqrt{\eps}=24\sqrt{\eps}\leq \gamma\]
    and we are done.
\end{proof}

\begin{remark}\label{RemarkContPointsCoarseSpaces}
    Notice that the inclusion in Theorem \ref{Thm.Quasi-LocalContainedContinuityPoints} is in general strict, even for u.l.f.\ coarse spaces.  Indeed, let $\cE_{\max}$ be the maximal u.l.f.\ coarse structure on $\N$, i.e., 
    \begin{align*}E&\in \cE_{\max}\ \Leftrightarrow\ \\
    &\sup_{n\in\N}\max\{|\{m\in \N\mid (n,m)\in E\}|,|\{m\in \N\mid (m,n)\in E\}|\}<\infty.\end{align*}Then $h\colon (\N,\cE_{\max})\to \R$ is coarse if and only if $h$ is bounded.  But then $\cB(\ell_2(\N))^{\sigma_h}=\cB(\ell_2(\N))$ for all coarse maps $h\colon \N\to \R$. On the other hand,  $\cstql(\N,\cE_{\max})$ is strictly smaller than $\cB(\ell_2(\N))$.
\end{remark}

\begin{proof}[Proof of Theorem \ref{Thm.qla.h.Points.Cont.Substructure}]
We first notice that $\cstql(X,\cE_h)$ is contained in the algebra $\cstql(X,\cE)^{\sigma_h}$. Fix a contraction  $a\in \cstql(X,\cE_h)$. As $\cE_h\subseteq \cE$, $a\in \cstql(X,\cE)$. By Theorem \ref{Thm.Quasi-LocalContainedContinuityPoints},  $a$ is a continuity point of $\sigma_h$ and we are done.

In order to show that $\cstql(X,\cE)^{\sigma_h}$ is contained in $\cstql(X,\cE_h)$, Theorem \ref{thmA} says that it is enough to show that \[\cstql(X,\cE)\cap \mathrm C^s[X,d_{h}]\subseteq \cstql(X,\cE_h)\ \text{ for all } \ s>0,\]
where $\mathrm C^s[X,d_h]$ is given as in \eqref{Eq.21.sep.26.1}.
 Fix $s>0$ and $a$ in $\cstql(X,\cE)\cap \mathrm C^s[X,d_{h}]$. Fix $\eps>0$ and let us show that there is $E\in \cE_h$ such that $a$ is $(3\eps,E)$-quasi-local. As $a\in \cstql(X,\cE)$, pick $F\in \cE$ such that $a$ is $(\eps,F)$-quasi-local and let
\[E=F\cap \{(x,y)\in X^2\mid d_h(x,y)\leq 2s\}.\]
So, $E \in \cE_h$.

Suppose $A,B\subseteq X$ are such that $(A\times B)\cap E=\emptyset$ and let us estimate $\chi_Aa\chi_B$. Let $(A_i)_i$ and $(B_i)_i$ be partitions of $A$ and $B$, respectively, given by letting,  for each $i\in \Z$,
\[A_i=h^{-1}([(i-1)s,is))\cap A\ \text{ and } \ B_i=h^{-1}([(i-1)s,is))\cap B.\]

\begin{claim}
    Let $i,j\in\Z$ be such that $\chi_{A_i}a\chi_{B_j}\neq 0$. Then
    \[|i-j|\leq 1\ \text{ and }\ (A_i\times B_j)\cap F=\emptyset.\]
\end{claim}

\begin{proof}
As $\chi_{A_i}a\chi_{B_j}\neq 0$, there are  $x\in A_i$ and $y\in B_j$ with $\langle a\delta_y,\delta_x\rangle\neq 0$. As $a$ has $d_h$-propagation at most $s$, this forces $|h(x)-h(y)|\leq s$ which in turn gives  $|i-j|\leq 1$.

Suppose now the intersection in the claim is not empty and pick  $x\in A_i$ and $y\in B_j$ with  $(x,y)\in F$. As $(A\times B)\cap E=\emptyset$, $(x,y)\not\in E$ and, by the definition of $E$, we must have that $d_h(x,y)>2s$. In other words, $|h(x)-h(y)|>2s$. But this implies that $|i-j|>1$; a contradiction.
\end{proof}

By the previous claim, we can write
\begin{equation}\label{Eq.19.Aug.26.2}\chi_Aa\chi_B=\sum_{j=-1}^1\SOTh\sum_{i\in\Z}\chi_{A_i}a\chi_{B_{i+j}}.\end{equation}
Moreover, since for any given $j\in \{-1,0,1\}$, the operators $(\chi_{A_i}a\chi_{B_{i+j}})_i$ are all orthogonal to each other, we have that
\begin{equation}\label{Eq.19.Aug.26.1}\left\|\SOTh\sum_{i\in\Z}\chi_{A_i}a\chi_{B_{i+j}}\right\|\leq \sup_{i\in\Z}\left\|\chi_{A_i}a\chi_{B_{i+j}}\right\|
\end{equation}
for all $j\in \{-1,0,1\}$. Appealing to the claim above again, if $\chi_{A_i}a\chi_{B_{i+j}}$ is not zero, then $(A_i\times B_{i+j})\cap F=\emptyset$ and, as $a$ is $(\eps,F)$-quasi-local, we have
\[\|\chi_{A_i}a\chi_{B_{i+j}} \|\leq \eps.\]
So, \eqref{Eq.19.Aug.26.1} is at most $\eps$ and, by \eqref{Eq.19.Aug.26.2}, we conclude that $\|\chi_Aa\chi_B\|\leq 3\eps$ as desired.
\end{proof}

\section{Dynamical characterization of the quasi-local algebra}

We have seen in Theorem \ref{Thm.Quasi-LocalContainedContinuityPoints} that the quasi-local algebra of a  coarse space is always contained in the \cstar-algebra of continuity points $\CP(X)$. While $\CP(X)$ is in general larger than $\cstql(X)$ (see Remark \ref{RemarkContPointsCoarseSpaces}), we now show that this is not the case for u.l.f.\ metric spaces. Therefore, for u.l.f.\ metric spaces, we obtain a dynamical characterization of quasi-local algebras (Corollary \ref{Cor.cstql=CP.metric}).

\begin{proposition}\label{PropNotInQLThereisCoarseNotInCPFlow}
    Let $(X,d)$ be a u.l.f.\ metric space. If $a\not\in \cstql(X)$, then there is a coarse map $h\colon X\to \R$ such that $a\not\in \cB(\ell_2(X))^{ \sigma_h}$.
    \end{proposition}

\begin{proof}
    Fix $a\not\in \cstql(X)$. Then there is $\eps>0$ such that for all $r>0$ there are $A,B\subseteq X$ such that $d(A,B)>r$ and $\|\chi_Aa\chi_B\|\geq \eps$.
    Without loss of generality, we can also assume that $A$ and $B$ are finite.

\begin{claim}
    For all $r>0$ and all finite $F\subseteq X$, there are finite $A,B\subseteq X\setminus F$ such that $d(A,B)>r$ and $\|\chi_Aa\chi_B\|\geq \eps/2$.
\end{claim}

\begin{proof}
    As $F\subseteq X$ is finite,
     $\chi_Fa$ and $a\chi_F$ are compact. So, there is a finite $F'\subseteq X$ with $F\subseteq F'$ such that $\|\chi_Fa\chi_{X\setminus F'}\|<\eps/2$ and
   $\|\chi_{X\setminus F'}a\chi_{F}\|<\eps/2$.  Also, using that $F'$ is finite   and replacing $r$ by a larger number if necessary, we can assume that $d(x,y)<r$ for all $x,y\in F'$.  Let then $A,B\subseteq X $ be the finite sets given by our choice of $\eps$ above. So, $d(A,B)>r$ and $\|\chi_Aa\chi_B\|\geq \eps$. Since $d(x,y)<r$ for all $x,y\in F'$, either $A\subseteq X\setminus F'$ or $B\subseteq X\setminus F'$. Suppose the former happens and let $B'=B\setminus F$. Then \[\|\chi_{A}a\chi_{B\cap F}\|\leq \|\chi_{X\setminus F'}a\chi_{F}\|<\eps/2\]
   and, since $\|\chi_Aa\chi_B\|\geq \eps$,  we must have that $\|\chi_Aa\chi_{B'}\|\geq \eps/2$.

   If $B\subseteq X\setminus F'$, the proof proceeds analogously.
\end{proof}

By the previous claim and using that $X$ is u.l.f., we can select sequences $(A_n)_n$ and $(B_n)_n$ of disjoint finite subsets of $X$ such that \begin{enumerate}
    \item $d(A_n,B_n)>n$ for all $n\in\N$,
    \item $\|\chi_{A_n}a\chi_{B_n}\|\geq \eps/2$ for all $n\in\N$, and
    \item $d(A_{n}\cup B_n,A_{m}\cup B_m)> n+m$ for all distinct $n,m\in\N$.
\end{enumerate}
For each $n\in\N$, let $h_n\colon X\to \R$ be the map given by
\[h_n(x)=\max\{n-d(x,A_n),0\}\]
    for all $x\in X$. It is then immediate to check that the map $h\colon X\to \R$ defined pointwise as
    \[h=\sum_{n}h_n\]
    is coarse.
    Moreover, if $A=\bigcup_nA_n$ and $B=\bigcup_nB_n$, then $h\restriction A_n=n$ for all $n\in\N$ and $h\restriction B=0$.

    We are left to notice that $a\not\in \cB(\ell_2(X))^{\sigma_h}$. But this is clear since, as
    \begin{align*}
   \|\sigma_{h,t}(a)-a\|&=\|[(e^{it(h(x)-h(y))}-1)\langle a\delta_y,\delta_x\rangle]_{x,y} \|
   \\
&\geq \sup_n\|\chi_{A_n}[(e^{it(h(x)-h(y))}-1)\langle a\delta_y,\delta_x\rangle]_{x,y} \chi_{B_n}\|\\
   &
  = \sup_n\|(e^{itn} -1)\chi_{A_n}a\chi_{B_n}\|,
    \end{align*}
    we must have that
    \[\liminf_{t\to 0,\,t\neq 0}\|\sigma_{h,t}(a)-a\|\geq \eps.\]
    This concludes the proof.
\end{proof}

The next corollary contains part of the statement  of Theorem \ref{thmB}.

\begin{corollary}\label{Cor.cstql=CP.metric}
    Let $X$ be a u.l.f.\ metric space. Then, $\cstql(X)=\CP(X)$.
\end{corollary}

\begin{proof}
By Theorem \ref{Thm.Quasi-LocalContainedContinuityPoints}, we have that $\cstql(X)\subseteq \CP(X)$. By  Proposition \ref{PropNotInQLThereisCoarseNotInCPFlow}, we have the other inclusion in case $X$ is a u.l.f.\ metric space.
\end{proof}

\section{Dynamical characterization of uniform Roe algebras}

In this section, we prove a dynamical characterization of uniform Roe algebras of u.l.f.\ metric spaces. More precisely,  we obtain Theorem \ref{thmB}, which says that the uniform Roe algebra of a u.l.f.\ metric space  $X$ coincides with the closure of all operators which are analytic for $\sigma_h$ for all coarse maps $h\colon X\to \R$, i.e., $\cstu(X)=\mathrm{AP}(X)$. We point out that since both $\cstu(X)$ and $\mathrm{AP}(X)$ are defined as the norm closure of certain $^*$-algebras, it is natural to wonder if these $^*$-algebras are already the same; in other words, 
is finite propagation equivalent to analyticity for $\sigma_h$ for all coarse maps $h\colon X\to \R$? We do not prove this here. We show however in Theorem \ref{thm:roeentire} that  an operator has finite propagation if and only if it is analytic \emph{of exponential type} (see Definition \ref{Def.BandAnalytic.functions}).\\

We start with definitions. First, we recall   the definition of analyticity for operators with respect to diagonal flows and introduce analyticity of exponential type.

\begin{definition}\label{Def.BandAnalytic.functions}
Let $X$ be a set,    $a\in\cB(\ell_2(X))$,  $h\colon X\to \R$, and $\delta>0$. We say that $a$ is \emph{analytic on a strip of width $\delta$ for $\sigma_h$} if   there is a continuous map \[F\colon\{z\in\C\mid |\mathrm{Im}(z)|\leq\delta\}\to\cB(\ell_2(X))\] which is holomorphic on the interior of the strip  $\{z\in\C\mid |\mathrm{Im}(z)|\leq\delta\}$ and satisfies \[F(t)=\sigma_{h,t}(a)\ \text{ for all }\ t\in\R.\]
We say that  $a$ is \emph{analytic on a strip   for $\sigma_h$} if there is $\delta>0$ such that $a$ is analytic on a strip of width $\delta$ for $\sigma_h$. If the analytic extension  $F$ can be defined on the entire complex plane, we simply say that $a$ is \emph{analytic for $\sigma_h$}. In this case, we also say that $a$ is \emph{analytic on a strip of infinite width for $\sigma_h$}. Finally, if moreover there are $C,K>0$ such that the map $F$ above satisfies
\[\|F(z)\|\leq Ce^{K|\mathrm{Im}(z)|}\ \text{ for all }\ z\in \mathrm{dom}(F),\]
then the analyticity of $a$ for $\sigma_h$ is said to be of \emph{exponential type}.
\end{definition}

The notions of analyticity above are now used to define certain diagonal \cstar-algebras. We define here only the algebras which play a role in this section. See Section  \ref{Section.Dynamical.algebra.between.URA.QL} for more on this.

\begin{definition}\label{Def.AP.algebra.defi}
Let $X$ be a coarse space.  For a given coarse map $h\colon X\to \R$, let  $\AP[X,h]$ be the set of operators in $ \cB(\ell_2(X))$ which are analytic   for $\sigma_h$ and   $\AP_{\exp}[X,h]$ be the set of operators in $ \cB(\ell_2(X))$ which are analytic of exponential type   for $\sigma_h$. We define
    \[\AP(X)=\overline{\bigcap_{\substack{h\colon X\to \R\\ h\text{ is coarse}}} \AP[X,h]}^{\|\cdot\|}\ \text{ and }\ \AP_{\exp}(X)=\overline{\bigcap_{\substack{h\colon X\to \R\\ h\text{ is coarse}}} \AP_{\exp}[X,h]}^{\|\cdot\|}.\]
\end{definition}

Notice   that the set inside the closure in the definition of $\AP(X)$ is a  unital $^*$-subalgebra of $\cB(\ell_2(X))$. Indeed, given a coarse $h$ and appropriate extensions $F$ and $G$ for $a$ and $b$ on strips of widths $\delta\in (0,\infty]$ and $\delta'$, the maps $F+G$ and $FG$ are extensions for $a+b$ and $ab$ on the strip  of width $\min\{\delta,\delta'\}$; the map $z\mapsto F(\bar z)^*$ is an extension for $a^*$. Moreover, the continuity and the analyticity of $F$ and $G$ pass to these maps.   In particular,  $\AP(X)$ is a \cstar-algebra. A similar argument shows that $\mathrm{AP}_{\exp}(X)$ is also a \cstar-algebra.

The next simple proposition   shows that, whenever the maps $t\mapsto \sigma_{h,t}(a)$ have analytic extensions,   the extension is actually unique and its entry-wise coefficients are completely determined. This will serve not only as a tool to compute the entry-wise coordinates of analytic extensions but also as a tool to define analytic extensions.

 \begin{proposition}\label{Prop.xy.coordinates.ana.extension}
     Let $X$ be a set, $h\colon X\to \R$ a map, and $a\in \cB(\ell_2(X))$ be analytic on a strip for $\sigma_h$. Let $F$ witness this analyticity on the strip  $\{z\in \C\mid |\mathrm{Im}(z)|\leq \delta\}$ where $\delta>0$. Then, for each $z$ in this strip and each $x,y\in X$, we have
     \[\langle F(z)\delta_y,\delta_x\rangle=e^{iz(h(x)-h(y))}\langle a\delta_y,\delta_x\rangle.\]
 \end{proposition}

 \begin{proof}
     Fix   $x,y\in X$. Firstly, as $F$ extends $t\mapsto \sigma_{h,t}(a)\in \cB(\ell_2(X))$, it is clear that
      \[\langle F(t)\delta_y,\delta_x\rangle=e^{it(h(x)-h(y))}\langle a\delta_y,\delta_x\rangle\ \text{ for all }\ t\in \R.\]
Therefore, as the map
     \[z\in \C\mapsto e^{iz(h(x)-h(y))}\langle a\delta_y,\delta_x\rangle\] is analytic and it coincides with $\langle F(\cdot)\delta_y,\delta_x\rangle$ on $\R$, it must equal it on the domain of $F$.
 \end{proof}

\begin{proposition}\label{Prop.analitic.iff.finite.dh.prop}
    Let $X$ be a set, $h\colon X\to \R$ be a map, and $a\in\cB(\ell_2(X))$. The following are equivalent.

\begin{enumerate}
    \item \label{Item.Prop.analitic.iff.finite.dh.prop.1} The operator $a$ is   analytic   of exponential type for  $\sigma_h$.
    \item\label{Item.Prop.analitic.iff.finite.dh.prop.3}  The operator $a$ has finite $d_h$-propagation.
\end{enumerate}
\end{proposition}

\begin{proof}
\eqref{Item.Prop.analitic.iff.finite.dh.prop.1} $\Rightarrow$\eqref{Item.Prop.analitic.iff.finite.dh.prop.3}: If $a$ is analytic of exponential type  for $\sigma_h$, there are  $C,K>0$ and an analytic function  $F\colon \C\to \cB(\ell_2(X))$ extending  $t\in \R\mapsto\sigma_{h,t}(a)\in \cB(\ell_2(X))$ such that    \[\|F(z)\|\leq Ce^{K|\mathrm{Im}(z)|}\ \text{  for all }\ z\in \C.\] By Proposition \ref{Prop.xy.coordinates.ana.extension},  we have
\[|\langle F(z)\delta_y,\delta_x\rangle|=|e^{iz (h(x)-h(y))}\langle a\delta_y,\delta_x\rangle|=e^{-\mathrm{Im}(z )(h(x)-h(y))}|\langle a\delta_y,\delta_x\rangle|\]
for all $x,y\in X$. Therefore,
\begin{equation}\label{Eq.27.aug.26.1}e^{-\mathrm{Im}(z )(h(x)-h(y))}|\langle a\delta_y,\delta_x\rangle|\leq Ce^{K|\mathrm{Im}(z)|}\end{equation}
for all $z\in \mathrm{dom}(F)$ and all $x,y\in X$.

Suppose $x,y\in X$ are such that $\langle a\delta_y,\delta_x\rangle\neq 0$. Then, choosing $z=t+i\delta$ and letting $\delta$ approach both $\infty$ and $-\infty$, we get  that
\[d_h(x,y)=|h(x)-h(y)|\leq K,\]
i.e., $a$ has $d_h$-propagation at most $K$.

\eqref{Item.Prop.analitic.iff.finite.dh.prop.3} $\Rightarrow$\eqref{Item.Prop.analitic.iff.finite.dh.prop.1}:  Suppose now that $M=d_h\text{-}\propg(a)$ is finite and let us show $a$ is analytic of exponential type for $\sigma_h$. Firstly, notice that, by Proposition \ref{Prop.xy.coordinates.ana.extension}, if $F$ is an analytic map witnessing the analyticity of $a$ for $\sigma_h$, then
\[\langle F(z)\delta_y,\delta_x\rangle=e^{iz(h(x)-h(y))}\langle a\delta_y,\delta_x\rangle\]
for all $z\in \C$.
By linearity, this implies that
\[\langle F(z)\xi,\zeta\rangle=\sum_{x,y\in X}e^{iz(h(x)-h(y))}\,\langle a\delta_y,\delta_x\rangle\,\xi_y\overline{\zeta_x}\]
for all  $\xi=(\xi_x)_{x\in X},\zeta=(\zeta_x)_{x\in X}\in c_{00}(X)$ and all $z\in \C$. This equality is our  starting point to obtain the analyticity of $a$. Precisely,  for $\xi,\zeta\in c_{00}(X)$, let  $u_{\xi,\zeta}\colon \C\to \C$ be given by
\[
u_{\xi,\zeta}(z)=\sum_{x,y\in X}e^{iz(h(x)-h(y))}\,\langle a\delta_y,\delta_x\rangle\,\xi_y\overline{\zeta_x}
\ \text{ for all }\ z\in \C.\]
 Notice that, for $t\in \R$, a straightforward computation gives that
\begin{equation}\label{Eq.19.Aug.26.tarde.1}u_{\xi,\zeta}(t)=\langle\sigma_{h,t}(a)\xi,\zeta\rangle.
\end{equation} Hence, as   each $\sigma_{h,t}$ is norm preserving,  for every $\xi,\zeta \in c_{00}(X)$, we have that
\begin{equation}\label{fix2:eq.realbound}
 |u_{\xi,\zeta}(t)|\ \le\ \|a\|\,\|\xi\|\,\|\zeta\|\ \text{ for all }\ t\in \R .
\end{equation}

 Notice that each $u_{\xi,\zeta}$ is an entire function of exponential type. Indeed, for a given $\xi,\zeta\in c_{00}(X)$, $u_{\xi,\zeta}$   is   a finite linear combination of the maps of the form \[z\mapsto e^{iz(h(x)-h(y))}\langle a\delta_y,\delta_x\rangle,\] where $x,y\in X$. As the $d_h$-propagation of $a$ is $M$, either $\langle a\delta_y,\delta_x\rangle=0$, or $|h(x)-h(y)|\leq M$. So, $u_{\xi,\zeta}$   is   a finite linear combination of the maps of the form $z\mapsto e^{iz\lambda}$ for $|\lambda|\leq M$.  By the classical Phragm\'en--Lindel\"of theorem for such functions (see \cite[Theorem 6.2.4]{Boas}), \eqref{fix2:eq.realbound} can be improved and we have that
\begin{equation}\label{fix2:eq.complexbound}
|u_{\xi,\zeta}(z)|\ \le\ e^{M|\mathrm{Im}(z)|}\,\|a\|\,\|\xi\|\,\|\zeta\|
\qquad\text{for all }z\in\mathbb{C}.
\end{equation}

Fix now $z\in\mathbb{C}$ and consider the sesquilinear map   \[(\xi,\zeta)\in c_{00}(X)\times c_{00}(X)\mapsto u_{\xi,\zeta}(z)\in \C.\] By \eqref{fix2:eq.complexbound}, this map is bounded. Hence there is a unique $F(z)\in\mathcal{B}(\ell_2(X))$ such that
\begin{equation}\label{Eq.19.Aug.26.tarde.2}
\langle F(z)\xi,\zeta\rangle=u_{\xi,\zeta}(z)\end{equation}
for all $\xi,\zeta\in c_{00}(X)$. By \eqref{fix2:eq.complexbound}, we have
\begin{equation}\label{Eq.19.Aug.26.tarde.3}\|F(z)\|\le e^{M|\mathrm{Im}(z)|}\|a\| \ \text{ for all } \ z\in \C.
\end{equation}
As $c_{00}(X)$ is dense in $\ell_2(X)$, it follows from   \eqref{Eq.19.Aug.26.tarde.1} and \eqref{Eq.19.Aug.26.tarde.2} that   \[F(t)= \sigma_{h,t}(a)\ \text{ for all } \ t\in \R.\]

As $F$ is an extension of $\sigma_{h,\cdot }(a)$ to the whole plane, it remains to show that   $F$ is entire. For this, let $W$ be the subset of all functionals in $\cB(\ell_2(X))^*$   given by
\[w_{\xi,\zeta}\colon b\in\mathcal{B}(\ell_2(X))\mapsto\langle b\xi,\zeta\rangle\in \C,\] where $\xi,\zeta\in c_{00}(X)$. Then $W$ is a   separating subset of $\mathcal{B}(\ell_2(X))^*$ in the sense that for all nonzero $b\in \cB(\ell_2(X))$ there are $\xi,\zeta\in c_{00}(X)$ such that   $w_{\xi,\zeta}(b)\neq 0$.  Moreover, the compositions $w_{\xi,\zeta}\circ F$   are precisely the entire functions $u_{\xi,\zeta}$. As   $F$ is locally bounded (see \eqref{Eq.19.Aug.26.tarde.3}),    \cite[Theorem 3.1]{ArendtNikolski} gives that   $F$ is entire.
\end{proof}

\begin{remark}
It is worth pointing out that the implication \eqref{Item.Prop.analitic.iff.finite.dh.prop.3}$\Rightarrow$\eqref{Item.Prop.analitic.iff.finite.dh.prop.1}  of Proposition \ref{Prop.analitic.iff.finite.dh.prop} is much simpler if $a$ has finite $d$-propagation for a u.l.f.\ metric $d$ on $X$ and $h$ is coarse with respect to $d$. Indeed, in this scenario, if $a$ has finite $d$-propagation, then there are finitely many $E_1,\ldots, E_n\subseteq X\times X$ such that each $E_i$ is the graph of a partial translation\footnote{Given a metric space $(X,d)$, a partial translation $f$ of $X$ is a partial bijection of $X$, say $f\colon A\to B$ where $A,B\subseteq X$,  such that $\sup_{x\in A}d(x,f(x))<\infty$.} of $X$ and
\[\{(x,y)\in X\times X\mid d(x,y)\leq \mathrm{prop}(a)\}=E_1\sqcup\ldots\sqcup E_n.\]
Therefore, we can write $a=\sum_{i=1}^na_iv_{f_i}$ for  appropriate $a_i$'s in $\ell_\infty(X)$ and partial translations $f_i$'s.\footnote{If $f$ is a partial translation of $X$, $v_f$ is the partial isometry on $\ell_2(X)$ such that $v_f\delta_x=\delta_{f(x)}$ for all $x\in \dom(f)$ and $v_f\delta_x=0$ for all other $x$'s. Clearly, $v_f$ has finite propagation.} It is then easy to show that each $a_iv_{f_i}$ is analytic of exponential type for $\sigma_h$; this can be done by following the proof of    \cite[Proposition 1.4(2)]{BragaExel2023}.  Moreover, this weaker statement is in fact all we will use below. We chose to present the more elaborate statement of Proposition \ref{Prop.analitic.iff.finite.dh.prop}  since it seems to be the correct characterization.
\end{remark}

\begin{lemma}\label{lem:detect} Let $X$ be a u.l.f.\ metric space and let $(x_k)_k$ and $(y_k)_k$ be sequences in $X$ such that $\lim_kd(x_k,y_k)=\infty$.  Then there is an increasing sequence $(n_k)_k$ in $\N$ and a $1$-Lipschitz map $h\colon X\to\R$ such that \[|h(x_{n_{k}}) -h(y_{n_{k}})|\geq \frac{d(x_{n_k},y_{n_k})}{9}\ \text{ for all }\ k\in\N.\]
\end{lemma}

\begin{proof}
Suppose first that either $(x_k)_k$ or $(y_k)_k$ repeats infinitely many times. Passing to a subsequence if necessary, we can assume without loss of generality that  all $x_k$'s are constant, say $z=x_k$ for all $k\in\N$. Then the map $h=d(\cdot, z)$ is $1$-Lipschitz and
\[|h(x_k)-h(y_k)|=|d(x_k,z)-d(y_k,z)|=d(y_k,x_k)\]
for all $k\in\N$.

Suppose now neither   $(x_k)_k$ nor $(y_k)_k$ repeats infinitely many times. As $X$ is u.l.f., this implies in particular that both sequences leave all bounded subsets of $X$. Hence, going to subsequences, we  can assume that
\begin{equation}\label{Eq.dxkyk.1111}d\left(\{x_{k+1},y_{k+1}\},\bigcup_{\ell=1}^k\{x_\ell,y_\ell\}\right)\geq\max_{1\leq \ell\leq k}d(x_\ell,y_\ell)\ \text{ for all }\ k\in\N. \end{equation}
Going to a further subsequence, we can also assume that 
\begin{equation}\label{Eq.dxkyk.1} d(x_{k+1},y_{k+1})\geq 3d(x_k,y_k) \ \text{ for all }\ k\in\N.\end{equation}

Notice that, for all $\ell,k\in\N$ with $\ell< k$, either
 \[d(x_{k},\{x_\ell,y_\ell\}) \geq \frac{d(x_{k},y_{k})}{3}\  \text{ or }\  d(y_{k},\{x_\ell,y_\ell\}) \geq \frac{d(x_{k},y_{k})}{3}.\]
Indeed, assume the contrary holds for some $\ell,k\in\N$ with $\ell<k$ and pick $z,w\in \{x_\ell,y_\ell\}$ such that
 \[d(x_{k},z)<\frac{d(x_{k},y_{k})}{3}\ \text{ and }\ d(y_{k},w)<\frac{d(x_{k},y_{k})}{3}. \]
Then,
 \begin{align*}
    d(x_{k},y_{k})&\leq d(x_{k},z)+d(z,w)+d(w,y_{k})\\
    &<\frac{2}{3}d(x_{k},y_{k})+d(x_\ell,y_\ell)\\
    &\leq d(x_{k},y_{k});
\end{align*}
a contradiction. By Ramsey's theorem, passing to a further subsequence, we can assume without loss of generality that\begin{equation}\label{Eq.dxkyk.2}d(x_{k},\{x_\ell,y_\ell\}) \geq \frac{d(x_{k},y_{k})}{3}
\ \text{ for all }\ \ell,k\in\N\ \text{ with }\ \ell <k.
\end{equation}

For each $k\in\N$, let
\[h_{k}(x)=\max\left\{0, \frac{d(x_{k},y_{k})}{9}  -d(x,x_k)\right\}\]
for all $x\in X$. So,   each $h_k$   is $1$-Lipschitz and has support contained in $B(x_k,\tfrac{d(x_k,y_k)}{9})$. In particular, by \eqref{Eq.dxkyk.1111}, \eqref{Eq.dxkyk.1}, and \eqref{Eq.dxkyk.2}, the supports of $(h_{k})_k$  are all disjoint and $h_k(y_\ell)=0$ for all distinct $\ell,k\in \N$. Hence, letting   $h\colon X\to \R$ be given by
\[h(x)=\sum_{k\in\N}h_{k}(x),\]
the map $h$ is   $1$-Lipschitz and it satisfies
\begin{align*}|h(x_{k})-h(y_{k})|=\frac{d(x_k,y_k)}{9}  \end{align*}
for all $k\in\N$.
 \end{proof}

\begin{corollary}\label{cor:detect} Let $X$ be a u.l.f.\ metric space and suppose $a\in\B(\ell_2(X))$ has infinite propagation. Then there is a $1$-Lipschitz map $h\colon X\to\R$ such that $a$ has infinite $d_h$-propagation.
\end{corollary}

\begin{proof}
If  $a$ has infinite propagation, then there are   sequences $(x_k)_k$ and $(y_k)_k$ in $X$ such that
\begin{equation*}
a_{x_k,y_k}\neq 0\ \text{ for all }\ k\in \N\ \text{ and }\lim_kd(x_k,y_k)=\infty.
\end{equation*}
By Lemma \ref{lem:detect}, there is a $1$-Lipschitz map $h\colon X\to \R$ such that \[\lim_kd_h(x_{n_k},y_{n_k})=\infty.\] As $a_{x_k,y_k}\neq 0$ for all $k\in\N$, $a$ has infinite $d_h$-propagation.
\end{proof}

\begin{theorem}\label{thm:roeentire}
Let $X$ be a u.l.f.\ metric space and $a\in\B(\ell_2(X))$. The following are equivalent:
\begin{enumerate}
\item \label{Item.thm:roeentire.1} The operator $a$ has finite propagation.
\item \label{Item.thm:roeentire.2}   For every coarse $h\colon X\to\R$, $a$ is   analytic of  exponential type for $\sigma_h$.
\end{enumerate}
In particular,
$
\cstu(X) =\mathrm{AP}_{\exp}(X).
$
\end{theorem}

\begin{proof}
\eqref{Item.thm:roeentire.1}$\Rightarrow$\eqref{Item.thm:roeentire.2}: Let $r=\propg(a)$. If $h\colon X\to \R$ is coarse,    then $\omega_h(r)<\infty$ and we must have that $d_h$-$\propg(a)\leq \omega_h(r)<\infty$. By Proposition  \ref{Prop.analitic.iff.finite.dh.prop}, $a$ is analytic and of exponential type for $\sigma_h$.

\eqref{Item.thm:roeentire.2}$\Rightarrow$\eqref{Item.thm:roeentire.1}: Suppose towards a contradiction that $a$ has unbounded propagation. Then, Corollary \ref{cor:detect} produces a coarse map $h\colon X\to\R$  such that $a$ has infinite $d_h$-propagation. By Proposition \ref{Prop.analitic.iff.finite.dh.prop}, this implies that $a$ is not analytic of exponential type   for $\sigma_h$.
\end{proof}

Now that we have characterized finite propagation in terms of analyticity of exponential type, we proceed to show that, for the closure of the algebras of such operators, exponential type does not make a difference. For that, we start noticing that, with out loss of generality, we can assume our u.l.f. metric spaces to have \emph{grow at most exponentially}: 

\begin{definition}\label{Defi.Exp.Growth}
    Let $X$ be a u.l.f.\ metric space. For each $m\in\N\cup\{0\}$, let 
    \[N_X(m)=\sup_{x\in X}|B(x,m)|.\]
    We say that $X$ \emph{grows at most exponentially} if there is $\lambda>1$ such that \[N_X(m)\leq \lambda^m\ \text{ for all }\  m\in\N.\]
\end{definition}

\begin{example}\label{Defi.21.sep.26.1.qqq}
Cayley graphs or, more generally, large scale geodesic spaces clearly grow at most exponentially.  
\end{example}

As the next result shows, every u.l.f.\ metric space is coarsely equivalent to a metric space which grows at most exponentially.

\begin{proposition}\emph{(}\cite[Proposition 5.1]{DranishnikovGongLafforgueYu2002CanMathJ}{)}.Every uniformly locally finite metric space coarsely embeds into
a connected graph with vertices of degree at most three.\label{DranishnikovGongLafforgueYu2002CanMathJ.11}
\end{proposition}

Before proving Theorem \ref{thmB}, we need some preparatory lemmas.

\begin{lemma}\label{lem:volume}
Let $X$ be a u.l.f.\ metric space and $a\in\B(\ell_2(X))$. For $m\in\N\cup\{0\}$, consider \[\eta_a(m)=\sup\{|a_{xy}|\mid d(x,y)\ge m\}
\]
and suppose that  \begin{equation}\label{Eq.Sum.Convam.eta}\sum_{m=0}^\infty N_X(m+1)\eta_a(m)<\infty.
\end{equation} Then, for all   $k\in\N$ there is $b\in \cB(\ell_2(X))$ with finite propagation   such that \[\|a-b\|\le2\sum_{m>k}N_X(m+1)\eta_a(m).\]
In particular, $a\in \cstu(X)$.
\end{lemma}

\begin{proof}
As $X$ is u.l.f.,   for each $m\in\N\cup\{0\}$, there is a partition
\[\{(x,y)\in X\times X\mid d(x,y)\in [m,m+1) \}=\bigsqcup_{i=1}^{2N_X(m+1)}E^m_i\]
into graphs of partial bijections (this is standard, see for instance   \cite[Lemma 2]{Ozawa2023}). For each $m\in\N\cup\{0\}$ and $i\in \{1,\ldots, 2N_X(m+1)\}$, let $a_i^m$ be the operator obtained from $a$ by considering only its entries in $E^m_i$, i.e.,  \[a_i^m=\SOTh\sum_{(x,y)\in E^m_i}\chi_{\{x\}}a\chi_{\{y\}}.\]
So, $\|a^m_i\|\leq \eta_a(m)$ and, letting \[a^m=\sum_{i=1}^{2N_X(m+1)}a^m_i,\]
we have that \[\|a^m\|\leq 2N_X(m+1)\eta_a(m).\]
As the sum in \eqref{Eq.Sum.Convam.eta} converges, the sum $\sum_{m=0}^\infty a^m$ also converges in norm and, as coordinate-wise it converges to $a$, we have that
\[a=\sum_{m=0}^\infty a^m.\]
It is immediate then that
\[\left\|a-\sum_{m=0}^ka^m\right\|\leq 2\sum_{m>k}N_X(m+1)\eta_a(m)\]
and we are done.\end{proof}

\begin{lemma}\label{lem:allrates}
Let $X$ be a u.l.f.\ metric space and let $a\in \cB(\ell_2(X))$ be such that $a$ is analytic for $\sigma_h$ for all $1$-Lipschitz maps $h\colon X\to \R$.  Then \[\sup_{x,y\in X}e^{c\,d(x,y)}|\langle a\delta_y,\delta_x\rangle|<\infty\ \text{ for all }\ c>0.\]
\end{lemma}

\begin{proof}
Suppose towards a contradiction that the lemma fails and fix a witness $c>0$. Then there are sequences $(x_k)_k$ and $(y_k)_k$ in $X$ such that
\begin{equation}\label{Eq.lem.allrates.1}e^{cd(x_k,y_k)}|\langle a\delta_{y_k},\delta_{x_k}\rangle|\geq k \text{ for all }\  k\in\N.
\end{equation} As each of the entries of $a$ has absolute value at most $\|a\|$,  this forces   $\lim_kd(x_k,y_k)=\infty$. By Lemma \ref{lem:detect}, going to subsequences if necessary, we  pick a  $1$-Lipschitz function $h\colon X\to \R$ such that
\begin{equation}\label{Eq.lem:allrates.1.1.} |h(x_{k})-h(y_{k})|\geq \frac{d(x_k,y_k)}{9}\end{equation}
for all $k\in\N$.

By the hypothesis of the lemma, there is an analytic function $F\colon \C\to \cB(\ell_2(X))$ which extends $t\in \R\mapsto \sigma_{h,t}(a)\in \cB(\ell_2(X))$. For each $z\in \C$, each of the entries of such map must be given by
\[\langle F(z)\delta_y,\delta_x\rangle=e^{iz(h(x)-h(y))}\langle a\delta_y,\delta_x\rangle\]
(see Proposition \ref{Prop.xy.coordinates.ana.extension}).
As the absolute value of these coordinates cannot be more than $\|F(z)\|$, this implies that
\[e^{s|h(x)-h(y)|} |\langle a\delta_y,\delta_x\rangle|\leq \max\left\{\|F(is)\|,\|F(-is)\|\right\}\]
for all $s>0$ and all $x,y\in X$. Letting $s=9c$, the previous inequality together with \eqref{Eq.lem.allrates.1} and \eqref{Eq.lem:allrates.1.1.} gives that
\begin{align*}k\leq e^{9c|h(x_k)-h(y_k)|}|\langle a\delta_{y_k},\delta_{x_k}\rangle|\leq \max\{\|F(9ic)\|,\|F(-9ic)\|\}
\end{align*}
for all $k\in\N$; contradiction.
\end{proof}

\begin{proof}[Proof of Theorem \ref{thmB}]
The inclusion ``$\subseteq$''  is immediate from Theorem \ref{thm:roeentire}.

 Let us show the inclusion ``$\supseteq$''  also holds. Firstly, by Proposition \ref{DranishnikovGongLafforgueYu2002CanMathJ.11}, $X$ is coarsely equivalent to a u.l.f.\ metric space which grows at most exponentially. As $\mathrm{AP}(X)$ depends only on the coarse type of $X$, we can then assume that $X$ grows at most exponentially. So, pick $\lambda>1$ such that $N_X(m)\leq \lambda^m$ for all $m\in\N$. Let now $a\in \cB(\ell_2(X))$ be analytic for $\sigma_h$ for all coarse maps $h\colon X\to \R$. By Lemma \ref{lem:allrates} applied to $c=\log(\lambda)+1$, we obtain that
 \[C=\sup_{x,y\in X}e^{cd(x,y)}{|\langle a\delta_y,\delta_x\rangle|}<\infty.\]
Letting $\eta_a(m)$ be as in Lemma \ref{lem:volume}  for each $m\in\N\cup\{0\}$, this implies that  \[\eta_a(m)\leq Ce^{-cm}\ \text{ for all }\ m\in\N.\]
By our choice of $c$, this implies that
\[\sum_{m=0}^\infty N_X(m+1)\eta_a(m)<\infty.\]
By  Lemma \ref{lem:volume}, this shows that $a$ is in $\cstu(X)$.
\end{proof}

\begin{remark}  In the light of Theorem \ref{thm:roeentire}, it is natural to wonder if $\cstu(X)$ also equals $\mathrm{AP}_{\mathrm{strip}}(X)$,     the \cstar-subalgebra of $\cB(\ell_2(X))$ generated by the elements which are analytic on a strip for $\sigma_h$ for every coarse $h$.   Firstly, it is worth pointing out that the proof of Theorem \ref{thmB} fails to give that since, in order to use  Lemma \ref{lem:volume}, we need the
 imaginary parts of the elements in the domain of the extension $F$ to be arbitrarily large. Moreover, this is actually false in general and we show this in Theorems \ref{Thm.Exp.decay.Still.Contains} and \ref{Thm.Led.In.Band} below. Precisely, we show that, if $X$ is a coarse disjoint union of expander graphs, then $\mathrm{AP}_{\mathrm{strip}}(X)$ contains the product $\prod_n\mathrm M_n$.
\end{remark}

We finish this section with an immediate consequence of Theorems \ref{thm:roeentire} and \ref{thmB}.
 
 \begin{corollary}
     Let $X$ be a u.l.f.\ metric space. Then $\mathrm{AP}(X)=\mathrm{AP}_{\exp}(X)$.\qed
 \end{corollary}

\section{The asymptotics of the quasi-locality  modulus and subalgebras generated by it}\label{Section.QL.algebras}

 In this section, we move away from dynamics and instead we study variations of the quasi-local algebra given by different asymptotic conditions on the quasi-locality modulus of operators (see Definition \ref{Defi.QL.Decay.modulus}). Each asymptotic behavior assumption for these moduli will lead to a \cstar-algebra lying in between the uniform Roe and the quasi-local algebra. Moreover, in the case of expander graphs, this will produce a continuum of pairwise distinct \cstar-algebras. We must notice however that, despite their natural definitions, these  are \emph{metric} constructions and not \emph{coarse} ones; indeed, these moduli are not invariant under bijective  coarse equivalences. In Section \ref{Section.Dynamical.algebra.between.URA.QL}, we return to the topic   of dynamical \cstar-algebras and show how the results obtained here can be used to construct a dynamical \cstar-algebra strictly between the uniform Roe and the quasi-local algebra; as expected, this will be a purely coarse construction. \\

We start with a definition which quantifies asymptotically the quasi-locality of an operator.

\begin{definition}\label{Defi.QL.Decay.modulus}
    Let $(X,d)$ be a metric space and $a\in \cB(\ell_2(X))$. The \emph{quasi-locality modulus of $a$} is the map $\eps_a\colon [0,\infty)\to [0,\infty)$ given by
    \[\eps_a(r)=\sup\left\{\|\chi_Aa\chi_B\|\mid A,B\subseteq X \text{ with } d(A,B)\geq r\right\}\]
    for all $r\geq 0$.
\end{definition}

\begin{example}\label{Example.decay.uRa}
    Given a metric space $X$, an operator $a\in \cB(\ell_2(X))$ has finite propagation if and only if $\eps_a(r)=0$ for $r$ sufficiently large.
\end{example}

\begin{example}\label{Example.decay.ql}
    Given a metric space $X$, an operator $a\in \cB(\ell_2(X))$ is in the quasi-local algebra of $X$  if and only if $\lim_{r\to \infty}\eps_a(r)=0$.
\end{example}

Between the asymptotic behaviors of $\eps_a$ being eventually zero and the one of $\eps_a$ simply approaching zero, there are many other possible variants. We will study two kinds of behavior: (1) polynomial decay and (2) exponential decay. We start properly defining the former.

\begin{definition}\label{Defi.QLalpha}
    Let $(X,d)$ be a u.l.f.\ metric space and $\alpha\in (0,\infty)$.
    Given $a\in\cB(\ell_2(X))$, we say that $\eps_a$ has \emph{$\alpha$-polynomial decay} if $\eps_a(r)=O(r^{-\alpha})$, i.e., if there is $C>0$ such that \[\eps_a(r)\leq Cr^{-\alpha}\ \text{  for all }\ r>0.\] We define
    \[\mathrm{QL}_{\alpha}(X)=\overline{\left\{a\in \cB(\ell_2(X))\mid \eps_a(r)= O(r^{-\alpha})\right\}}^{\|\cdot\|}.\]
    \end{definition}

\begin{definition}\label{Defi.QL.exp}
    Let $(X,d)$ be a u.l.f.\ metric space.
 Given $a\in\cB(\ell_2(X))$, we say that $\eps_a$ has \emph{exponential decay} if  there are $c,C>0$ such that
    \[\eps_a(r)\leq Ce^{-cr}\ \text{ for all }\ r\geq 0.\] We define
    \[\mathrm{QL}_{\exp}(X)=\overline{\left\{a\in \cB(\ell_2(X))\mid \eps_a\text{ decays exponentially}\right\}}^{\|\cdot\|}.\]
\end{definition}

In  view of Definitions \ref{Defi.QLalpha} and \ref{Defi.QL.exp}, some comments are in order.
Notice that, for $a,b\in \cB(\ell_2(X))$,  \[\eps_a=\eps_{a^*},\ \eps_{a+b}\leq \eps_a+\eps_b,\ \text{ and }\ \eps_{ab}(2r)\leq \eps_{a}(r)\|b\|+\eps_b(r)\|a\|.\] So, the elements of $\cB(\ell_2(X))$ with $\alpha$-polynomial decay, for  $\alpha>0$, and the ones with exponential decay form  $^*$-subalgebras of $\cB(\ell_2(X))$. So, $\mathrm{QL}_\alpha(X)$ and   $\mathrm{QL}_{\exp}(X)$ are \cstar-algebras. Moreover, by Examples \ref{Example.decay.uRa} and \ref{Example.decay.ql},
\[\cstu(X)\subseteq \mathrm{QL}_{\exp}(X) \subseteq \mathrm{QL}_{\beta}(X)\subseteq \mathrm{QL}_\alpha(X)\subseteq\cstql(X)\]
for all  $\alpha,\beta>0$ with  $\alpha<\beta$. For the remainder of this section, we show that  all inclusions above can be strict  (Theorems \ref{Thm.Exp.decay.Still.Contains} and \ref{Thm.Inclusion.QL.Algebras}).

\subsection{Embeddability of product of matrix algebras}
In this subsection, we show that if $X$ is the coarse disjoint union of expander graphs, then the product of matrix algebras $\prod_n\mathrm M_n$ embeds into $\mathrm{QL}_{\exp}(X)$ (Theorem \ref{Thm.Exp.decay.Still.Contains}). As $\mathrm{QL}_{\exp}(X)
\subseteq \cstql(X)$, this is a strengthening of the result of Ozawa showing that  $\prod_n\mathrm M_n$  embeds into $\cstql(X)$ (see \cite[Theorem B]{Ozawa2023}).

The proof of the embeddability of $\prod_n\mathrm M_n$ in $\mathrm{QL}_{\exp}(X)$ uses the   concentration of measure phenomenon. The first draft of this paper had the complete arguments for that. Luckily, a recent preprint \cite{LiZhangZhu2026} takes care of the gory details and the next lemma can be obtained as a corollary of their work.  In order not to break the flow of this piece, we make the didactic choice of postponing the proof of the next lemma to  Appendix \ref{Appendix}. 

Given $n\in\N$ and a subspace $W\subseteq \C^n$, we let $p_W$ denote the orthogonal projection $\C^n\to W$.  The next lemma is a modification of   \cite[Lemma 4]{Ozawa2023}.

\begin{lemma}\emph{(Proved in  Appendix \ref{Appendix})}.  There is a universal constant $C\geq 1$ such that, for all $n\in\N$, there is a   subspace $W\subseteq\C^n$ with $\dim(W)=\lceil n^{1/4}\rceil$ such that
\[
\|\chi_Ap_W\|\leq C\sqrt{\delta\log\Big(\frac1\delta\Big)}
\]
for all $\delta\in[\dim(W)^{-1},1/2]$ and all $A\subseteq\{1,\dots,n\}$ with $|A|\leq n\delta$.\label{LemmaGoodSubspacesQuarterPower}
\end{lemma}

\begin{theorem}\label{Thm.Exp.decay.Still.Contains}
    Let $X$ be a coarse disjoint union of expander graphs. Then $\prod_n\mathrm M_n$ embeds into $\mathrm{QL}_{\exp}(X)$.
\end{theorem}

\begin{proof}
  Write  $X=\bigsqcup_nX_n$, where $(X_n)_n$ is a sequence of expander graphs.  Fix $\kappa>1$ such that
   \begin{equation}\label{Eq.asympt.exp.eq}\min\left\{\frac{|A|}{|X_{n}|},\frac{|B|}{|X_{n}|}\right\}\leq \kappa^{ -\frac{d(A,B)}{2}}\end{equation}
  for all $n\in\N$ and all $A,B\subseteq X_n$ (see Subsection \ref{Subsection.Coarse.Disj.Union}).

  For each $n\in\N$, let $W_n\subseteq \C^{|X_n|}$ be the   subspace given by Lemma \ref{LemmaGoodSubspacesQuarterPower} for $n=|X_n|$ therein. Identify each $\C^{|X_n|}$ with $\ell_2(X_n)$.  We now gather the properties of these spaces given by this lemma. For each $n\in\N$, let
  \[m_n=\dim(W_n).\]
In particular,
\[\lim_nm_n=\lim_n\left\lceil |X_n|^{1/4}\right\rceil=\infty.\]
By  the conclusion of Lemma \ref{LemmaGoodSubspacesQuarterPower} and letting $C\geq 1$ be the universal constant therein, we have that
\begin{equation}\label{Eq.GoodSubspace.Complex}
        \|\chi_Ap_{W_n}\|\leq C\sqrt{\delta\log\left(\frac{1}{\delta}\right)}
    \end{equation}
  for all $n\in\N$, all $\delta\in[1/m_n,1/2]$, and all $A\subseteq X_n$ with $|A|\leq \delta |X_n|$.

  Fix a contraction
  \[a=\SOTh\sum_na_n\in \prod_n\cB(W_n)\subseteq \cB(\ell_2(X))\]
  and let us show that
   \begin{equation}\label{Eq.UnifExpDecay}
   \eps_a(r)\leq 2C\kappa^{-r/24}\ \text{ for all }\ r>0.
   \end{equation}
  In particular, this will imply that $a\in\mathrm{QL}_{\exp}(X)$. For that, fix $r>0$ and let $A,B\subseteq X$ be such that $d(A,B)\geq r$. For each $n\in\N$, let
 \[A_n=A\cap X_{n}\ \text{ and }\ B_n=B\cap X_{n}.\]
  As $a$ is in $\prod_n\cB(\ell_2(X_n))$,  we have
 \[\|\chi_Aa\chi_B\|=\sup_n\|\chi_{A_n}a_n\chi_{B_n}\| .\]
So, we are left to estimate each
$\|\chi_{A_n}a_n\chi_{B_n}\|$.

Without loss of generality, we can assume   that  $\kappa^{-r/2}\leq 1/2$. Indeed,   if $\kappa^{-r/2}>1/2$, then $\kappa^{-r/24}>2^{-1/12}$ and  the right-hand side of \eqref{Eq.UnifExpDecay} is larger than $2C\cdot 2^{-1/12}$.  Since the latter is larger than $1$, it follows immediately that \eqref{Eq.UnifExpDecay}  holds for such $r$'s. From now on, we assume
\begin{equation}\label{Eq.21.aug.26.1}
\kappa^{-r/2}\leq 1/2.
\end{equation}

 Fix $n\in\N$. If either $A_n$ or $B_n$ is empty, then $\chi_{A_n}a\chi_{B_n}$ is zero and we are done. Otherwise, notice that $m_n\geq \kappa^{r/8}$. Indeed, fix $x\in A_n$ and $y\in B_n$. Then $d(x,y)\geq r$ and, by  \eqref{Eq.asympt.exp.eq}, we have that
 \[\frac{1}{|X_n|}\leq \kappa^{-\frac{d(x,y)}{2}}\leq \kappa^{-\frac{r}{2}}.\]
 Therefore, $\kappa^{r/2}\leq |X_n|$ and, as $m_n\geq |X_n|^{1/4}$, we obtain that \begin{equation}\label{Eq.s.uva.1}
 m_n\geq  \kappa^{r/8},
 \end{equation}
 as desired. Applying \eqref{Eq.asympt.exp.eq} again, we can assume without loss of generality that  \begin{equation}\label{Eq.21.aug.26.2}|A_n|\leq \kappa^{-r/2}|X_n|.\end{equation}

Define
 \[\delta=\max\left\{\kappa^{-r/2},\frac{1}{m_n}\right\}.\] By \eqref{Eq.21.aug.26.1},   $\delta\in[1/m_n,1/2]$ and, by \eqref{Eq.21.aug.26.2},  $|A_n|\leq \delta|X_n|$. Therefore, estimate \eqref{Eq.GoodSubspace.Complex} gives that
\begin{align}\label{Eq.21.Aug.26.lb.1}
\|\chi_{A_n}a_n\chi_{B_n}\|\leq \|\chi_{A_n}a_n\|\leq \|\chi_{A_n}p_{W_n}\| \leq C\sqrt{\delta\log\left(\frac1\delta\right)},
\end{align}
where here we are using that, since $a_n\in \cB(W_n)$, we have  $a_n=p_{W_n}a_n$. In order to estimate the last term in \eqref{Eq.21.Aug.26.lb.1}, recall that $\log(t)\leq 3t^{1/3}$ for $t\geq 1$. Then, as $\delta\leq 1$,
\begin{align}\label{Eq.21.Aug.26.lb.2}
\|\chi_{A_n}a_n\chi_{B_n}\|
\leq 2C\delta^{1/3}.\end{align}
Finally, from the definition of $\delta$, we have $\delta\leq  \kappa^{-r/8}$. Indeed, if   $\delta=\kappa^{-r/2}$, there is nothing to be said and, if $\delta=1/m_n$, this follows from  \eqref{Eq.s.uva.1}. Putting this all together, we conclude that
\[\|\chi_{A_n}a_n\chi_{B_n}\|\leq 2C\kappa^{-r/24}\]
as desired.\end{proof}

\subsection{Strict inclusions of algebras of operators with polynomial decay moduli}

The following is the main result of this section.

\begin{theorem}\label{Thm.Inclusion.QL.Algebras}
    Let $X$ be a coarse disjoint union of expander graphs. Then
    \begin{equation}
        \mathrm{QL}_{\exp}(X)\subsetneq \mathrm{QL}_\beta(X)\subsetneq \mathrm{QL}_\alpha(X)  \ \text{ for all }\ \alpha,\beta>0\ \text{ with }\ \alpha<\beta. \end{equation}
\end{theorem}

In order to prove Theorem \ref{Thm.Inclusion.QL.Algebras}, we will once again rely on the concentration of measure phenomenon. As in the previous subsection, we do not present the proof of the precise consequence of this phenomenon we need here. Instead, we postpone the details to Appendix \ref{Appendix}. Recall, given a subspace $W\subseteq \C^n$, $p_W$ denotes the orthogonal projection $\C^n\to W$.

\begin{lemma}\emph{(Proved in Appendix \ref{Appendix}).}  There is a universal constant $C\geq 1$ such that for all $n\in \N$ and $\alpha>0$ with $1\leq(\log(n))^\alpha<n$, there is a subspace $W\subseteq\C^n$  with $\dim(W)=\lfloor \tfrac{n}{(\log{n})^\alpha}\rfloor$ such that
\[
\|\chi_Ap_W\|\leq C\sqrt{\frac{1}{(\log{n})^\alpha}+\frac{|A|}{n}\log\left(\frac{en}{|A|}\right)}
\]
for all nonempty $A\subseteq \{1,\ldots, n\}$.\label{LemmaGoodSubspacesQuarterPower.LARGEDIM}
\end{lemma}

The strict inclusions in Theorem \ref{Thm.Inclusion.QL.Algebras} will be obtained by showing that certain projections, which we call $p_\alpha$, are in $\mathrm{QL}_\alpha(X)$ but not in $\mathrm{QL}_\beta(X)$ as long as $\alpha<\beta$. The definition of these projections is rather technical and it makes use of Lemma \ref{LemmaGoodSubspacesQuarterPower.LARGEDIM}. For simplicity of the statements which follow, we shall isolate our hypothesis now and define these projections properly.

\begin{assumption}\label{Assumption.1}
    Let $X$ be a u.l.f.\ metric space which is the  coarse disjoint union of connected finite graphs, say $X=\bigsqcup_n X_n $. Assume that $\lim_n|X_n|=\infty$. Fix $\alpha>0$ and $n_0\in\N$ such that $n\geq n_0$ implies $1\leq(\log(|X_n|))^ {2\alpha}<|X_n|$. For each $n\geq n_0$, let $p_{n,\alpha}$ be a projection in $\cB(\ell_2(X_n))$ such that $\mathrm{rank}(p_{n,\alpha})$ is $\left\lfloor \tfrac{|X_n|}{(\log{(|X_n|)})^ {2\alpha}}\right\rfloor$ and
    \[
\|\chi_Ap_{n,\alpha}\|\leq C\sqrt{\frac{1}{(\log{(|X_n|)})^{2\alpha}}+\frac{|A|}{|X_n|}\log\left(\frac{e|X_n|}{|A|}\right)}
\]
for all nonempty $A\subseteq X_n$; the existence of such $p_{n,\alpha}$ and of such $C\geq 1$ is given by Lemma \ref{LemmaGoodSubspacesQuarterPower.LARGEDIM}. Define then
\[p_\alpha=\SOTh\sum_{n\geq n_0}p_{n,\alpha}.\]
So, $p_\alpha$ is a projection in $\prod_n\cB(\ell_2(X_n))$.
\end{assumption}

We now prove the easier part of Theorem \ref{Thm.Inclusion.QL.Algebras}: we show that $p_\alpha$ is in $\mathrm{QL}_\alpha(X)$ as long as $X$ in Assumption \ref{Assumption.1} is a coarse disjoint union of expander graphs. The proof is inspired by \cite[Theorem B]{Ozawa2023}.

\begin{proposition}\label{Prop.palpha.is.inQLalpha.exp}
   Assume the  setting of Assumption \ref{Assumption.1} and moreover that  $X=\bigsqcup_nX_n$ is a coarse disjoint union of expander graphs. Then,  $p_\alpha$ is in $\mathrm{QL}_\alpha(X)$.
\end{proposition}

\begin{proof}
We start fixing $\kappa>1$ such that for all  $n\in\N$ and all $A,B\subseteq X_n$, we have
\[\min\left\{\frac{|A |}{|X_n|}, \frac{|B|}{|X_n|}\right\}\leq \kappa^{-\frac{d(A,B)}{2}}.\]
Applying this to two singletons in each $X_n$ realizing the diameter of $X_n$, we get that ${1}/{|X_n|}$ is at most  $\kappa^{-\mathrm{diam}(X_n)/2}$, i.e., 
\begin{equation}\label{Eq.t.t.e.tb}\mathrm{diam}(X_n)\leq \frac{2\log |X_n|}{\log(\kappa)}.\end{equation}

Fix now $r>0$ and $A,B\subseteq X$ with $d(A,B)\geq r$. For each $n\in\N$, let \[A_n=A\cap X_n\ \text{ and }\ B_n=B\cap X_n.\] Since
\[
\|\chi_Ap_\alpha\chi_B\|=\sup_{n\geq n_0}\|\chi_{A_n}p_{n,\alpha}\chi_{B_n}\|,
\]
we only need to estimate each $\|\chi_{A_n}p_{n,\alpha}\chi_{B_n}\|$.
If either $A_n$ or $B_n$ is empty,  this operator is zero. So, we  assume otherwise from now on.

  As $A_n$ and $B_n$ are nonempty subsets of $X_n$ with $d(A_n,B_n)\geq r$, we have $r\leq\mathrm{diam}(X_n)$. By \eqref{Eq.t.t.e.tb},  
\begin{equation}\label{eq:diam}
\frac{1}{(\log(|X_n|))^{2\alpha}}\leq\frac{2^{2\alpha}}{(\log(\kappa))^{2\alpha}r^{2\alpha}}
\end{equation}
for all $n\geq n_0$.

By our choice of $\kappa$,  we can assume without loss of generality that   $|A_n|/|X_n|\leq\kappa^{-r/2}$; otherwise this inequality holds with $B_n$ instead of $A_n$ and the proof proceeds analogously. As the map $t\mapsto t\log(e/t)$ is increasing on $(0,1]$,
\begin{equation}\label{Eq.ddqd.q}
\frac{|A_n|}{|X_n|}\log\left(\frac{e|X_n|}{|A_n|}\right)\leq\kappa^{-\frac r2}\log\left(e\kappa^{\frac r2}\right)=\left(1+\frac{r\log\kappa}{2}\right)\kappa^{-\frac r2}.
\end{equation}
By Assumption \ref{Assumption.1},    \eqref{eq:diam} and \eqref{Eq.ddqd.q} give that
\[\|\chi_{A_n}p_{n,\alpha}\|\leq C\sqrt{\frac{2^{2\alpha}}{(\log(\kappa))^{2\alpha}r^{2\alpha}}+ \left(1+\frac{r\log\kappa}{2}\right)\kappa^{-\frac r2}}.\]
In particular, there is $C'>0$ independent of $r$ such that
\[\|\chi_{A_n}p_{n,\alpha}\|\leq C'r^{-\alpha}\]
for all $n\geq n_0$. As $\|\chi_{A_n}p_{n,\alpha}\chi_{B_n}\|\leq \|\chi_{A_n}p_{n,\alpha}\|$, this shows that $\eps_{p_\alpha}(r)\leq C'r^{-\alpha}$ as desired.  \end{proof}

We are left to show that $p_\alpha$ does not belong to $\mathrm{QL}_\beta(X)$ if $\alpha<\beta$; this will conclude the proof of Theorem \ref{Thm.Inclusion.QL.Algebras}. We start with two lemmas; the first of them is a simple statement about traces of operators. In words, we know that traces of projections are simply their ranks. The lemma below gives an explicit approximation for the trace of a self-adjoint operator approximating  a projection.

\begin{lemma}\label{lem:trace}
Let $H$ be a finite dimensional Hilbert space, $\delta\in (0,1)$, $p\in\B(H)$ be a projection, and $a\in\B(H)$ be self-adjoint with $\|a-p\|\leq \delta$. Then
\[
\mathrm{Tr}(a^{2m})\geq \rank(p)(1-\delta)^{2m}\ \text{for all}\  m\in\N.
\]
\end{lemma}

\begin{proof}
Let $E$ be the spectral measure given by the Spectral Theorem for the positive  operator $a^2$ (see \cite{rudin1991functional}). Then, for a unit vector $\xi\in \mathrm{ran}(p)$, we have
\[\langle a^{2k}\xi,\xi\rangle=\int_{\sigma(a^2)} t^kdE_{\xi,\xi}(t)\] for all $k\in\N$. Since \[\|a\xi\|\geq\|p\xi\|-\|(a-p)\xi\|\geq 1-\delta\] and $t\mapsto t^m$ is convex on $\sigma(a^2)\subseteq [0,\infty)$, Jensen's inequality gives
\begin{align*}
\langle a^{2m}\xi,\xi\rangle&=\int_{\sigma(a^2)} t^m\,dE_{\xi,\xi}(t)\\
&\geq\left(\int_{\sigma(a^2)} t\,dE_{\xi,\xi}(t)\right)^m\\
&=\langle a^2\xi,\xi\rangle^m\\
&=\|a\xi\|^{2m}\\
&\geq(1-\delta)^{2m}.
\end{align*}
Letting $(e_i)_{i=1}^{\dim H}$ be an orthonormal basis of $H$ whose first $\rank(p)$ elements form a basis of $\mathrm{ran}(p)$,  this gives that
\begin{align*}
\mathrm{Tr}(a^{2m})&=\sum_{i=1}^{\dim H}\langle a^{2m}e_i,e_i\rangle\\
&\geq\sum_{i=1}^{\rank(p)}\langle a^{2m}e_i,e_i\rangle\\
&\geq\rank(p)(1-\delta)^{2m}.
\end{align*}
This finishes the proof.\end{proof}

For the next lemma, we use the following standard notation: given a metric space $(X,d)$, $x\in X$, and $r>0$, we write 
\[B(x,r)=\{y\in X\mid d(x,y)\leq r\}.\]
    \begin{lemma}
\label{lem:compression}
Let $X$ be a metric space,   $a\in\B(\ell_2(X))$, $x\in X$, $r>0$, $m\in\N$, and let $A=B(x,m r)$. Then
\[
\|a^m\delta_x\|\leq\|\chi_Aa\chi_A\|^m+m\|a\|^{m-1}\varepsilon_a(r).
\]
    \end{lemma}

 \begin{proof}
For each $j\in\{0,\dots,m\}$, let $A_j=B(x,j r)$; so $A_0=\{x\}$ and $A_m=A$. Define $(\xi_j)_{j=0}^m$ recursively by letting  \[\xi_0=\delta_x \ \text{ and }\ \xi_j=\chi_{A_j}a\xi_{j-1}\] for all $j\in\{1,\dots,m\}$. As each $\xi_{j-1}$ is supported in $A_{j-1}\subseteq A$, we have
\[
\xi_j=\chi_{A_j}a\chi_{A_{j-1}}\xi_{j-1}=\chi_{A_j}\big(\chi_Aa\chi_A\big)\chi_{A_{j-1}}\xi_{j-1}.
\]
Hence,   $\|\xi_j\|\leq\|\chi_Aa\chi_A\|\,\|\xi_{j-1}\|$ and  $\|\xi_j\|\leq\|a\|\,\|\xi_{j-1}\|$, and, by induction,
\begin{equation}\label{eq:xi}
\|\xi_m\|\leq\|\chi_Aa\chi_A\|^m\quad\text{and}\quad\|\xi_j\|\leq\|a\|^j
\end{equation}
for all $j\in \{0,\ldots, m\}$.
Therefore,
\begin{equation}\label{Eq.t.w.m.d.o.a}
    \|a^m\delta_x\|\leq \|\xi_m\|+\|a^m\delta_x-\xi_m\|\leq \|\chi_Aa\chi_A\|^m+\|a^m\delta_x-\xi_m\|
\end{equation}
and we are left to estimate $\|a^m\delta_x-\xi_m\|$.

A telescoping sum gives
\[
a^m\delta_x-\xi_m=\sum_{j=1}^m\big(a^{m-j+1}\xi_{j-1}-a^{m-j}\xi_j\big)=\sum_{j=1}^m a^{m-j}\big(a\xi_{j-1}-\xi_j\big).
\]
Since $\xi_{j-1}$ is supported in $A_{j-1}$,
\[
a\xi_{j-1}-\xi_j=\chi_{X\setminus A_j}a\xi_{j-1}=\chi_{X\setminus A_j}a\chi_{A_{j-1}}\xi_{j-1}
\]
and we have
\begin{equation}\label{eq:xi.1}
a^m\delta_x-\xi_m=\sum_{j=1}^m a^{m-j}\chi_{X\setminus A_j}a\chi_{A_{j-1}}\xi_{j-1}.
\end{equation}
Notice that, if $y\in A_{j-1}$ and $z\in X\setminus A_j$, then \[d(y,z)\geq d(x,z)-d(x,y)>j r-(j-1)r=r.\] Hence $d(X\setminus A_j,A_{j-1})\geq r$ and $\|\chi_{X\setminus A_j}a\chi_{A_{j-1}}\|\leq\varepsilon_a(r)$. By \eqref{eq:xi} and \eqref{eq:xi.1}, this gives that
\begin{align*}
\|a^m\delta_x-\xi_m\|&\leq\sum_{j=1}^m\|a\|^{m-j}\,\|\chi_{X\setminus A_j}a\chi_{A_{j-1}}\|\,\|\xi_{j-1}\|\\
&\leq\sum_{j=1}^m\|a\|^{m-j}\varepsilon_a(r)\|a\|^{j-1}\\
&=m\|a\|^{m-1}\varepsilon_a(r).
\end{align*}
By \eqref{Eq.t.w.m.d.o.a}, this completes the proof of the lemma.
    \end{proof}

\begin{theorem}
    \label{Thm.QLthetaDoesNotContainP} In the setting of Assumption \ref{Assumption.1}, $p_\alpha$ is not in $\mathrm{QL}_{\beta}(X)$ for any $\beta\in (\alpha,\infty)$.
\end{theorem}

\begin{proof}
    Fix $\beta\in (\alpha,\infty)$. Let us show that $p_\alpha$ is not in $\mathrm{QL}_{\beta}(X)$, i.e., that $p_\alpha$ is not in the closure of the $^*$-algebra of all operators $a\in \cB(\ell_2(X))$ such that $\eps_a(r)= O(r^{-\beta})$. For this, fix such $a$    and let us estimate $\|a-p_\alpha\|$. Recall that, by the definition of $p_\alpha$ in Assumption \ref{Assumption.1}, 
    \[p_\alpha=\sum_{n\geq n_0}p_{n,\alpha},\]
    where $n_0$ is selected precisely in the assumption. 
     We start   noticing that there is no loss of generality in assuming that $a$ is self-adjoint and that \begin{equation}\label{Eq.a.formula.self.ad.blocks}
    a=\SOTh\sum_{n\geq n_0}a_n.\end{equation}
    Indeed, replacing $a$ by $a'=\SOTh\sum_{n\geq n_0}\chi_{X_n}b\chi_{X_n}$, where $b=(a+a^*)/2$, we would  have that \[\|a'-p_\alpha\|\leq \|a-p_\alpha\| \ \text{ and }\ \eps_{a'}\leq \eps_a.\] So, we assume from now on that $a$ is self-adjoint and that \eqref{Eq.a.formula.self.ad.blocks} holds. Notice that
    \[\|a_n-p_{n,\alpha}\|\leq \|a-p_\alpha\| \ \text{ and }\ \eps_{a_n}(r)\leq \eps_a(r)\]
for all $n\geq n_0$ and $r>0$.

    Let $\delta\in(0,1/3)$ with $\|a-p_\alpha\|\leq\delta$ and let us show that $\delta$ cannot be arbitrarily small. This will be done by contradiction. Precisely, we show that if $\delta$ can be taken arbitrarily small, then $\beta\leq \alpha$.  The proof will proceed by finding, for each $n\geq n_0$, a lower and an upper estimate for the trace of powers of the operators $a_n$ and noticing that, for $n$ sufficiently large, these estimates force $\beta$ to be at most $\alpha$.  We start with the lower estimate. For that, given $n\geq n_0$ and  $m\in\N$, Lemma \ref{lem:trace} applied to the   $p_{n,\alpha}$ and $a_n$ gives that
    \begin{equation}\label{Eq.Trace.Lower.bound}\mathrm{Tr}\left(a_n^{2m}\right)\geq \mathrm{rank}(p_{n,\alpha})(1-\delta)^{2m}\geq \frac{|X_n|}{2(\log(|X_n|))^ {2\alpha}}(1-\delta)^{2m},\end{equation}
    here we use the condition on the rank of $p_{n,\alpha}$ in Assumption \ref{Assumption.1}. This is the lower estimate we shall need for each $n\geq n_0$.

We now turn to finding an upper estimate for the left-hand side of \eqref{Eq.Trace.Lower.bound}. As
    \[\mathrm{Tr}(a_n^{2m})=\sum_{x\in X_n}\langle a_n^{2m}\delta_x,\delta_x\rangle=\sum_{x\in X_n}\|a_n^m\delta_x\|^2,\]
this reduces to finding upper bounds for each   $\|a_n^m\delta_x\|$. As $\eps_a(r)= O(r^{-\beta})$,   increasing $C$ if necessary, we have that  \[\eps_a(r)\leq Cr^{-\beta} \ \text{ for all }\ r>0.\]
Then, as  \[\|a_n\|\leq \|p_{n,\alpha}\|+\|p_{n,\alpha}-a_n\|\leq 1+\delta,\]  Lemma \ref{lem:compression}
implies that, for each $n\geq n_0$ and $m\in\N$, $r>0$, and $x\in X_n$, we have
\begin{equation}\label{Eq.qweqweqwe}\|a^m_n\delta_x\|\leq\|\chi_Aa_n\chi_A\|^m+m(1+\delta)^{m-1}Cr^{-\beta},\end{equation}
where $A=B(x,mr)\cap X_n$.

In order to estimate \eqref{Eq.qweqweqwe}, we will now have to choose appropriate $m$ and $r$  for each $n\in\N$; we shall denote each of them by $m_n$ and $r_n$, respectively. We start this with a claim which estimates the first term in the right-hand side of \eqref{Eq.qweqweqwe} and provides the first condition on $m_n$ and $r_n$ in terms of the cardinality of the set \begin{equation}\label{Eq.SetAn}A_{n,x}=B(x,m_nr_n)\cap X_n.\end{equation}
For this,   fix $k\in\N$ with $k\geq2$ to be such that for each $n\in\N$, each element in $X_n$ is connected to at most $k$ other elements of $X_n$ --- such $k$ exists as $X$ is u.l.f.

\begin{claim}\label{Claim.ddd.q}
For  $n\geq n_0$ sufficiently large   and   $A\subseteq X_n$   nonempty with $|A|\leq k\sqrt{|X_n|}$, we have \[
\|\chi_Ap_{n,\alpha}\chi_A\|\leq C^2\left(\frac{1}{(\log(|X_n|))^ {2\alpha}}+\frac{k\log(e|X_n|)}{\sqrt{|X_n|}}\right).
\]
\end{claim}

\begin{proof}
 As $|A|/|X_n|\leq k/\sqrt{|X_n|}\leq 1$ and the function $t\mapsto t\log(e/t)$ is increasing on $(0,1]$, we have
\[
\frac{|A|}{|X_n|}\log\left(\frac{e |X_n|}{|A|}\right)\leq\frac{k}{\sqrt {|X_n|}}\log\left(\frac{e\sqrt{|X_n|}}{k}\right)\leq\frac{k\log(e |X_n|)}{\sqrt{|X_n|}}.
\]
Hence, by Assumption \ref{Assumption.1},
\begin{align*}
\|\chi_Ap_{n,\alpha}\chi_A\| &=\|\chi_Ap_{n,\alpha}\|^2\\
&\leq C^2\left(\frac{1}{(\log (|X_n|))^ {2\alpha}}+\frac{k\log(e|X_n|)}{\sqrt{|X_n|}}\right)
\end{align*}
and we are done.
\end{proof}

As $\lim_n|X_n|=\infty$, we can fix $n_1\in\N$ such that $n\geq n_1$ implies
\begin{equation}\label{Choiceofn}C^2\left(\frac{1}{(\log(|X_n|))^ {2\alpha}}+\frac{k\log(e|X_n|)}{\sqrt{|X_n|}}\right)\leq \delta.\end{equation}
In order to use the estimate in Claim \ref{Claim.ddd.q} for the sets $A_{n,x}$ in \eqref{Eq.SetAn},  we must choose, for each $n\geq n_1$, $m_n$ and $r_n$ such that, for all $x\in  X_n$, we have  $|A_{n,x}|\leq k\sqrt{|X_n|}$. For this, first notice that
\[|B(x,mr)\cap X_n|\leq 1+k+\ldots+k^{\lfloor mr\rfloor}\leq k^{mr+1}\]
for all $x\in X_n$, $m\in \N$ and $r>0$.
So, in order to apply  Claim \ref{Claim.ddd.q} for each of the sets $A_{n,x}$, it is sufficient to choose $m_n\in\N$ and $r_n>0$ such that
\begin{equation*} k^{m_nr_n}\leq \sqrt{|X_n|},\end{equation*}
or, equivalently, such that
\begin{equation}\label{Eq.mnrn.k.claim}{m_nr_n}\leq\frac{ \log(|X_n|)}{2\log(k)}.\end{equation}
For such $m_n$ and $r_n$, Claim \ref{Claim.ddd.q} and our choice of $n_1$ (see \eqref{Choiceofn}) give that
\begin{equation}\label{Eq.Consequence.Claim.mnrn}\|\chi_{A_{n,x}}a_n\chi_{A_{n,x}}\|\leq \|\chi_{A_{n,x}}(a_n-p_{n,\alpha})\chi_{A_{n,x}}\|+\|\chi_{A_{n,x}}p_{n,\alpha}\chi_{A_{n,x}}\|\leq 2\delta\end{equation}
for all $x\in X_n$.

Fix $n\geq n_1$ for now. Assume for the time being that $m_n\in\N$ has been chosen and let
\begin{equation}\label{Eq.Choice.of.rn.1}r_n=\left(m_n\left(\frac{1+\delta}{\delta}\right)^{m_n}C\right)^{1/\beta};\end{equation}
we will give the precise definition of $m_n$  later and show that, for this choice of $r_n$,   \eqref{Eq.mnrn.k.claim} is satisfied.  Noticing that  $r_n$ was chosen precisely so that
\begin{equation}\label{Eq.t.q.c.a.aa}m_n(1+\delta)^{m_n-1}Cr_n^{-\beta}\leq  \delta^{m_n},\end{equation}
it follows from  \eqref{Eq.qweqweqwe}, \eqref{Eq.Consequence.Claim.mnrn},  and \eqref{Eq.t.q.c.a.aa}  that
\[\|a^{m_n}_n\delta_x\|\leq 2 (2\delta)^{m_n}.\]
Therefore
\[\sum_{x\in X_n}\|a^{m_n}_n\delta_x\|^2\leq 4|X_n|(2\delta)^{2m_n}\]
and, by  \eqref{Eq.Trace.Lower.bound}, we have
\[\frac{|X_n|}{2(\log(|X_n|))^{2\alpha}}(1-\delta)^{2m_n}\leq 4|X_n|(2\delta)^{2m_n}. \]
Rearranging and simplifying the equation above, we have
\begin{equation}\label{Eq.1.sep.26.1.rain}\left(\frac{1-\delta}{2\delta}\right)^{2m_n}\leq 8(\log(|X_n|))^{2\alpha}.\end{equation}

Let now $\beta'\in (\alpha,\beta)$ and define
\[m_n=\left\lfloor \frac{\beta'\log(\log(|X_n|))}{\log\left(\frac{1+\delta}{\delta}\right)}\right\rfloor.\]
Suppose for now that $m_n$ and $r_n$ satisfy \eqref{Eq.mnrn.k.claim} and let us complete the proof. Indeed, it follows from a simple computation that the  left-hand side of
\eqref{Eq.1.sep.26.1.rain} becomes at least
\[\left(\frac{2\delta}{1-\delta}\right)^2(\log(|X_n|))^{2\beta'\frac{\log\left(\frac{1-\delta}{2\delta}\right)}{\log\left(\frac{1+\delta}{\delta}\right)}}.\]
As the term multiplying $2\beta'$ in the exponent above tends to $1$ as $\delta\to 0^ +$, \eqref{Eq.1.sep.26.1.rain} implies that $\beta'\leq \alpha$. As $\beta'$ was an arbitrary element in $(\alpha,\beta)$, we conclude that $\beta\leq \alpha$; a contradiction.

We are left to notice that for $\delta$ small enough and $n$ large enough, our choice of $m_n$ and $r_n$   satisfies \eqref{Eq.mnrn.k.claim}. We add the details here, but this is nothing but a simple estimate. Indeed, we start by fixing   $\delta_0>0$  such that
\[ \left(\frac{\beta'}{\log\left(\frac{1+\delta}{\delta}\right)}\right)^{1+1/\beta}\leq \frac{1}{2\log(k)}\]
for all $\delta\in (0,\delta_0)$. Increasing our choice of $n_1$ if necessary, we can also assume that
\[ C^{1/\beta}(\log(\log(|X_n|)))^{1+1/\beta}\leq (\log(|X_n|))^{1-\beta'/\beta}\]
for all $n\geq n_1$. Finally, since
\[\left(\frac{1+\delta}{\delta}\right)^{\frac{\beta'\log(\log(|X_n|))}{\beta \log\left(\frac{1+\delta}{\delta}\right)}}= (\log(|X_n|))^{\beta'/\beta},\]
this shows that
\[m_nr_n\leq \frac{\log(|X_n|)}{2\log(k)}\]
for all $\delta\in (0,\delta_0)$ and $n\geq n_1$.
\end{proof}

  \begin{proof}
      [Proof of Theorem \ref{Thm.Inclusion.QL.Algebras}]
      This follows from Proposition \ref{Prop.palpha.is.inQLalpha.exp} and Theorem \ref{Thm.QLthetaDoesNotContainP}.
  \end{proof}

\begin{proof}
    [Proof of Theorem \ref{thmD}]
    By Theorem \ref{Thm.Inclusion.QL.Algebras}, we have that 
    \[\mathrm{QL}_{\exp}(X)\subsetneq \mathrm{QL}_{\beta}(X)\subsetneq \mathrm{QL}_{\alpha}(X)\]
    for all $\alpha,\beta>0$ with $\alpha<\beta$. As $\mathrm{QL}_{\beta}(X)\subseteq\cstql(X)$ for all $\beta>0$, this gives that 
        \[\mathrm{QL}_{\exp}(X)\subsetneq \mathrm{QL}_{\beta}(X)\subsetneq \mathrm{QL}_{\alpha}(X)\subsetneq \cstql(X)\]
    for all $\alpha,\beta>0$ with $\alpha<\beta$. Finally, since $\prod_n\mathrm M_n$ does not embed into $\cstu(X)$ (see \cite[Theorem A]{Ozawa2023}) but it does embed into $\mathrm{QL}_{\exp}(X)$ (see Theorem \ref{Thm.Exp.decay.Still.Contains}), we conclude that 
    \[\cstu(X)\subsetneq \mathrm{QL}_{\exp}(X)\subsetneq \mathrm{QL}_{\beta}(X)\subsetneq \mathrm{QL}_{\alpha}(X)\subsetneq \cstql(X)\]
    for all $\alpha,\beta>0$ with $\alpha<\beta$ as desired.
\end{proof}

\section{A dynamical \cstar-algebra between the uniform Roe algebra and the quasi-local algebra}
\label{Section.Dynamical.algebra.between.URA.QL}
By Theorem \ref{thmB}, we have dynamical characterizations of $\cstu(X)$ and $\cstql(X)$ in terms of the regularity of the elements in $\cB(\ell_2(X))$ with respect to all coarse maps $h\colon X\to \R$. Precisely, $\cstql(X)$ is characterized as the points which are continuous for all such flows, and $\cstu(X)$ by the norm closure of the ones which are analytic  for all such flows. Between these two regularity conditions, there are of course many other conditions (e.g., Lipschitz, different orders of differentiability)  and they can be used to define dynamical \cstar-subalgebras (formally) in between $\cstu(X)$ and $\cstql(X)$.

We shall now study a dynamical  \cstar-subalgebra which is given by weakening the regularity condition which characterizes $\cstu(X)$. Moreover,   we show that this  algebra is strictly in between $\cstu(X)$ and $\cstql(X)$ for expander graphs. Contrary to the \cstar-algebras studied in Section \ref{Section.QL.algebras}, this algebra depends purely on the coarse geometry of $X$.

\begin{definition} \label{Defi.AP_band}
Let $X$ be a coarse space. For a given coarse map $h\colon X\to \R$, let $\AP_{\mathrm{strip}}[X,h]$ denote the  operators in $\cB(\ell_2(X))$ which are analytic on a strip for $\sigma_h$. We define
\[\AP_{\mathrm{strip}}(X)=\overline{\bigcap_{\substack{h\colon X\to \R\\ h\text{ is coarse}}}\AP_{\mathrm{strip}}[X,h]}^{\|\cdot\|}.\]
\end{definition}

For the same reasons as the ones presented in the paragraph immediately after Definition \ref{Def.AP.algebra.defi}, the   set inside the closure in the definition of $\AP_{\mathrm{strip}}(X)$   is a unital $^*$-subalgebra of $\cB(\ell_2(X))$ and, in particular,  $\AP_{\mathrm{strip}}(X)$ is a \cstar-algebra.

Together with the results in Section \ref{Section.QL.algebras}, the main result in this section shows that $\AP_{\mathrm{strip}}(X)$ sits strictly in between $\cstu(X)$ and $\cstql(X)$ whenever $X$ is a coarse disjoint union of expander graphs. For this, we will show that $\mathrm{AP}_{\mathrm{strip}}(X)$ is always contained in  $\mathrm{QL}_{\exp}(X)$ and that the converse inclusion holds for coarse disjoint unions of finite graphs (see Theorems \ref{Thm.Band.In.Led} and \ref{Thm.Led.In.Band}).

 \subsection{$\mathrm{AP}_{\mathrm{strip}}(X)$ is contained in  $\mathrm{QL}_{\exp}(X)$}
The following is the main result of this subsection.

\begin{theorem}\label{Thm.Band.In.Led}
Let $X$ be a u.l.f.\ metric space and let $a\in\cB(\ell_2(X))$ be analytic on a strip for $\sigma_h$ for every $1$-Lipschitz $h\colon X\to\R$. Then $\eps_a$ decays exponentially. In particular,
\[\mathrm{AP}_{\mathrm{strip}}(X)\subseteq \mathrm{QL}_{\exp}(X).\]
\end{theorem}

 We need two preparatory lemmas before the proof of Theorem \ref{Thm.Band.In.Led}.

\begin{lemma}\label{Lemma.BandStructure}
Let $X$ be a set, $h\colon X\to\R$,    $a\in\cB(\ell_2(X))$ be analytic on a strip for $\sigma_h$, and let   $F$ witness this analyticity in the strip  $\{z\in \C\mid |\mathrm{Im}(z)|\leq \delta\}$, where   $\delta>0$.
 Then
 \[F(t+is)=\sigma_{h,t}(F(is))\] for all  $t\in\R$  and all  $\ |s|\leq\delta$. In particular,
    \[\sup\big\{\|F(z)\|\ \big|\ |\mathrm{Im}(z)|\leq\delta\big\}=\max_{|s|\leq\delta}\|F(is)\|<\infty.\]
\end{lemma}

\begin{proof}
By Proposition \ref{Prop.xy.coordinates.ana.extension}, for each $x,y\in X$, we have
\begin{align*}
\langle     F(t+is)\delta_y,\delta_x\rangle &=e^{i(t+is)(h(x)-h(y))}\langle a\delta_y,\delta_x\rangle\\
    &=e^{it (h(x)-h(y))}e^{-s(h(x)-h(y))}\langle a\delta_y,\delta_x\rangle\\
    &=\langle \sigma_{h,t}(F(is))\delta_y,\delta_x\rangle;
\end{align*}
hence the equality $F(t+is)=\sigma_{h,t}(F(is)) $ follows.

As  each $\sigma_{h,t}$ is an isometry, the equality just proven implies that the supremum  of the norm of $F$ happens in any segment $\{t+is\mid |s|\leq \delta\}$, $t\in \R$. In particular, this is so for $t=0$ and the final equality   follows.
\end{proof}

\begin{lemma}\label{Lemma.GapEstimate}
Let $X$ be a set, $h\colon X\to\R$, $\theta>0$, and $a\in\cB(\ell_2(X))$. Suppose the $X$-by-$X$ matrix \[\left[e^{\theta(h(x)-h(y))}\langle a\delta_y,\delta_x\rangle\right]_{x,y\in X}\] induces a bounded operator $b$ on $\ell_2(X)$. Then, for all nonempty $A,B\subseteq X$, we have
\[\|\chi_Aa\chi_B\|\leq e^{-\theta\left( \inf_{x\in A}h(x)- \sup_{x\in B}h(x)\right)}\|b\|.\]
\end{lemma}

\begin{proof}
Write  $\alpha=\inf_{x\in A}h(x)$ and $\beta=\sup_{x\in B}h(x)$. Let $c_A,c_B\in \ell_\infty(X)$ be given by
\[c_A(x)=\left\{\begin{array}{cc}
 e^{-\theta(h(x)-\alpha)}  ,  & x\in A, \\
  0   & x\not\in A,
\end{array}\right.\] and \[  c_B(x)=\left\{\begin{array}{cc}
 e^{-\theta(\beta-h(x))} ,  & x\in B, \\
  0   & x\not\in B.
\end{array}\right.\] Then, comparing matrix entries, we have
\[c_A\left(\chi_Ab\chi_B\right)c_B=e^{\theta(\alpha-\beta)}\chi_Aa\chi_B.\]
As both $c_A$ and $c_B$ are contractions,   the result follows.
\end{proof}

\begin{proof}[Proof of Theorem \ref{Thm.Band.In.Led}]
For each $z\in X$, the maps $\pm d(\cdot ,z)$ are $1$-Lipschitz. By the hypothesis on $a$, for each $z\in X$, fix $\theta_z>0$ such that $a$ is analytic for both $\sigma_{  d(\cdot ,z)}$ and $\sigma_{-d(\cdot ,z)}$ on the strip $\{\lambda\in \C\mid |\mathrm{Im}(\lambda)|\leq \theta_z\}$. By Lemma \ref{Lemma.BandStructure}, we can pick $K_z>0$ such that the $X$-by-$X$ matrices
\[\left[e^{\pm\theta_z( d(x,z)-d(y,z))}\langle a\delta_{y},\delta_{x}\rangle\right]_{x,y\in X}\]
induce bounded operators of norm at most $K_z$. We will show that, for  appropriately chosen $n_0\in\N$ and finite   $G\subseteq X$, the constants
\begin{equation}\label{Eq.C.c.bwb.pedro}C=n_0+2\sum_{z\in G}K_z\ \text{ and }\ c=\min\left(\left\{\frac{1}{2n_0}\right\}\cup\{\theta_z\mid z\in G\}\right)\end{equation} satisfy that   $\eps_a(r)\leq Ce^{-cr}$ for all $r\geq 0$. Notice that, as $G$ is finite, $c$ is indeed strictly positive. The precise $n_0$ and $G$ will be chosen below.

Fix a base point $x_0\in X$ and let
\[L=\left\{h\colon X\to\R\mid h\text{ is $1$-Lipschitz and }h(x_0)=0\right\}.\]
Endow $L$ with the topology of pointwise convergence and notice this makes $L$ into a compact Hausdorff space. Indeed, Hausdorffness  is immediate and compactness follows by Tychonoff's theorem since $|h(x)|\leq d(x,x_0)$ for all $h\in L$. In particular, we can make use of Baire's theorem  for $L$. With that in mind, let us write $L$ as a countable union of closed sets.

 For each $h\in L$ and $n\in\N$, let $b_{h,n}$ be the $X$-by-$X$ matrix
\[\left[e^{\frac{h(x)-h(y)}{n}}\langle a\delta_y,\delta_x\rangle\right]_{x,y\in X}.\]
Notice that we are not claiming here that $b_{h,n}$ induces a bounded operator on $\ell_2(X)$. Let
\[E_n=\left\{h\in L\mid \|b_{h,n}\|\leq  n\right\}.\]

\begin{claim}
    Each $E_n$ is closed in $L$ and $L=\bigcup_nE_n$.
\end{claim}

\begin{proof}
Let $n\in\N$. If $(h_i)_i$ is in $E_n$ and it converges to some $h\in L$ with $\|b_{h,n}\|>n$, then there is a finite $I\subseteq X$ such that      $\|\chi_Ib_{h,n}\chi_I\|>n$. As $(h_i)_i$ converges pointwise to $h$ and $I$ is finite, there is $i$ large enough such that  $\|\chi_Ib_{h_i,n}\chi_I\|>n$; a contradiction since $h_i\in E_n$. So, $E_n$ is closed.

Let now $h\in L$ and let us notice  that $h\in E_n$ for some $n\in\N$. As $a$ is analytic on a strip for $h$, pick $F$ which witnesses that on a strip of width $\delta>0$. By Proposition \ref{Prop.xy.coordinates.ana.extension}, for all $n\geq 1/\delta$ and all $x,y\in X$, we have
\[\left\langle F\left(-\frac{i}{n}\right)\delta_y,\delta_x\right\rangle=e^{\frac{h(x)-h(y)}{n}}\langle a\delta_y,\delta_x\rangle=\langle b_{h,n}\delta_y,\delta_x\rangle.\]
Hence, by Lemma \ref{Lemma.BandStructure},
\[\|b_{h,n}\|=\left\|F\left(-\frac{i}{n}\right)\right\|\leq \sup_{s\in [-\delta,\delta]}\|F(is)\|.\]
Therefore, if $K$ is the supremum in the right-hand side above, then
\[\|b_{h,n}\|\leq n\]
for all $n\in\N$ with $n\geq \max\{1/\delta,K\}$. So, $h$ belongs to $E_n$ for any such $n$.
\end{proof}

  By Baire's category theorem, there is $n_0\in\N$ such that $E_{n_0}$ has nonempty interior. Fix $h_0\in L$, a finite $I\subseteq X$ containing $x_0$, and $\eta>0$  such that
\[\left\{h\in L\mid |h(x)-h_0(x)|<\eta\ \text{ for all }x\in I\right\}\subseteq E_{n_0}.\]
 Let $\Delta=\mathrm{diam}(I)$ and set
\[G=\{x\in X\mid d(x,I)< 3\Delta\}.\]
As $X$ is u.l.f., $G$ is finite as required. Let $c$ and $C$ be as in \eqref{Eq.C.c.bwb.pedro} for such $G$ and let us show that $\eps_a(r)\leq Ce^{-cr}$ for all $r\geq 0$.

Fix  $r>0$  and let   $A,B\subseteq X$ be nonempty such that   $d(A,B)\geq r$. We write
\begin{equation}\label{Eq.first.second.third.summand}\chi_Aa\chi_B=\chi_{A\setminus G}\,a\,\chi_{B\setminus G}+\chi_{A\cap G}\,a\,\chi_{B}+\chi_{A\setminus G}\,a\,\chi_{B\cap G}\end{equation}
and estimate the norm of each of the three terms in the right-hand side above. We start with the latter two terms.   For the second term, for each  $x\in A\cap G$,   let $h_x=-d(\,\cdot\,,x)$; so,     $h_x(x)=0$ and $h_x(y)\leq -r$  for all $y\in B$. Therefore,  applying Lemma \ref{Lemma.GapEstimate} for the sets $\{x\}$ and $B$, for $\theta=\theta_x$ and $h=h_x$, we get that \begin{equation*} \|\chi_{\{x\}}a\chi_B\|\leq K_xe^{-\theta_xr}.\end{equation*} Summing over $x\in A\cap G$, this implies that \begin{equation}\label{Eq.28.aug.26.2} \|\chi_{A\cap G}a\chi_B\| \leq  \sum_{x\in G}K_xe^{-\theta_xr}.\end{equation}
The third summand in  \eqref{Eq.first.second.third.summand} is  estimated analogously. Precisely, letting  $h_x=d(\cdot,x)$ for $x\in B\cap G$, a similar argument implies that  \begin{equation}\label{Eq.28.aug.26.3}\|\chi_{A}a\chi_{B\cap G}\| \leq  \sum_{x\in G}K_xe^{-\theta_xr}.\end{equation} 

We are left to estimate the first term in the right-hand side of \eqref{Eq.first.second.third.summand}. We start with a claim.

\begin{claim}\label{Claim.g.hg.cg}
   For every $1$-Lipschitz $g\colon X\to[0,\infty)$ there is $h_g\in E_{n_0}$ and $c_g\in\R$ such that
\begin{equation}\label{Eq.Gluing}
h_g(x)=\frac{g(x)}{2}+c_g\ \text{ for all }\ x\in X\setminus G .
\end{equation}
\end{claim}

\begin{proof}
Let \[c_g=\max_{z\in F}\left(h_0(z)-\frac{g(z)}{2}\right)\] and
\begin{equation}\label{Eq.Defi.hg}h_g=\min\Big\{\min_{z\in F}\big(h_0(z)+d(\cdot,z)\big),\ \frac{g}{2}+c_g\Big\}.\end{equation}
As $h_g$ is clearly $1$-Lipschitz, we only need to notice that $h_g$ is in $E_{n_0}$ and that it equals $g/2+c_g$ on $X\setminus G$. For the former, notice that  $h_g(x)=h_0(x)$ for all $x\in F$. Indeed, for a fixed $x\in F$, it is clear that the first term in the minimum in \eqref{Eq.Defi.hg} is at most $h_0(x)$. So, $h_g(x)\leq h_0(x)$.  On the other hand,
\begin{align*}\frac{g(x)}{2}+c_g&= \frac{g(x)}{2}+ \max_{z\in F}\left(h_0(z)-\frac{g(z)}{2}\right)\\
&\geq h_0(x)\end{align*}
and, as $h_0$ is $1$-Lipschitz, 
\[h_0(z)+d(x,z)\geq h_0(x)\]
for all $z\in F$.  So, $h_g(x)\geq h_0(x)$.

We now show that  $h_g$  equals $g/2+c_g$ on $X\setminus G$.
 For this fix  $x\in X\setminus G$  and let $z\in F$ be such that  $c_g=h_0(z)-g(z)/2$. Then, for  $y\in F$, using that $g$ and $h_0$ are $1$-Lipschitz, and that $d(x,y)\geq3\Delta$, we have
\begin{align*}
\frac{g(x)}{2}+c_g&=\frac{g(x)-g(z)}{2}+h_0(z)\\
&\leq \frac{d(x,z)}{2}+h_0(y)+\Delta\\
&\leq \frac{d(x,y)+\Delta}{2}+h_0(y)+\Delta\\
&\leq h_0(y)+d(x,y).
\end{align*}
This shows that the minimum in \eqref{Eq.Defi.hg} is  attained in its second element for $x$, i.e., $h_g(x)=g(x)/2+c_g$.
\end{proof}

For the first summand in \eqref{Eq.first.second.third.summand}, let $g=d(\,\cdot\,,B)$ and  pick  $h_g\in E_{n_0}$ as in Claim \ref{Claim.g.hg.cg}. Then,
\[h_g(x)=c_g\ \text{ for all }\ x\in B\setminus G\] and, as $g(x)\geq d(A,B)$ for all $x\in A$ and $d(A,B)\geq r$,   \[h_g(x)\geq r/2+c_g\ \text{ for all }\ x\in A\setminus G.\]
Applying  Lemma \ref{Lemma.GapEstimate}  with $\theta=1/n_0$ and $h=h_g$,    we obtain
\begin{equation}\label{Eq.28.aug.26.1}\|\chi_{A\setminus G} a\chi_{B\setminus G}\|\leq n_0e^{-\frac{r}{2n_0}}.\end{equation}

At last, by  \eqref{Eq.28.aug.26.2},   \eqref{Eq.28.aug.26.3}, and \eqref{Eq.28.aug.26.1}, we conclude that \[\|\chi_Aa\chi_B\|\leq Ce^{-cr},\] as desired.
\end{proof}

\subsection{$\mathrm{QL}_{\exp}(X)$ is  contained in  $\mathrm{AP}_{\mathrm{strip}}(X)$ for coarse disjoint unions of finite graphs}
 The next theorem is proved in this subsection.

\begin{theorem}\label{Thm.Led.In.Band}
Let $X$ be a u.l.f.\ metric space which is a coarse disjoint union of   finite connected graphs. Then
\[\mathrm{QL}_{\exp}(X)\subseteq \AP_{\mathrm{strip}}(X).\]
\end{theorem}

We need some preliminary results, all of them of a technical nature. For that, we first introduce some notation.  Let $X$ be a set, $a\in \cB(\ell_2(X))$, $h\colon X\to \R$, and $k\in  \N$. If the $X$-by-$X$ matrix
\[\left[(h(x)-h(y))^k\langle a\delta_y,\delta_x\rangle\right]_{x,y\in X}\]
induces a bounded operator on $\ell_2(X)$, we denote this operator by $\ad_h^k(a)$.
For $k=0$, we   let   $\ad^0_h(a)=a$. Notice that, for 
$k\in \N$, 
\begin{equation}\label{Eq.defi.adk.inductive}
\ad^k_h(a)=\ad_h\left(\ad^{k-1}_h(a)\right)=\left[h,\ad^{k-1}_h(a)\right].
\end{equation}

The next lemma is an integral formula which, in particular, shows that operators $a$ with $\ad_h(a)$ bounded have Lipschitz orbits.

\begin{lemma}\label{Lemma.Duhamel}
Let $X$ be a set,  $h\colon X\to\R$ be a map, and  $a\in\cB(\ell_2(X))$ be such that $\ad_h(a)$ is bounded. Then
\[\|\sigma_{h,t}(a)-a\|\leq |t|\,\|\ad_h(a)\|\ \text{ for all }\ t\in\R.\]
\end{lemma}

\begin{proof}
First notice that, for each $x,y\in X$ and $t\in\R$, we have
\begin{align}\label{Eq.28.aug.26.dj.1}
 \langle (\sigma_{h,t}(a)-a)\delta_y,\delta_x\rangle &=(e^{it(h(x)-h(y))}-1)\langle a\delta_y,\delta_x\rangle\\
 &=\int_0^te^{is(h(x)-h(y))} i(h(x)-h(y)) ds\cdot\langle a\delta_y,\delta_x\rangle\notag\\
 &=\int_0^te^{is(h(x)-h(y))} \langle  i(h(x)-h(y))a\delta_y,\delta_x\rangle ds
 \notag\\
 &=\int_0^t\langle \sigma_{h,s}(i\ad_h(a))\delta_y,\delta_x\rangle ds.\notag
\end{align}
Moreover, by linearity, \eqref{Eq.28.aug.26.dj.1} holds if the vectors $\delta_y$ and $\delta_x$ are replaced by  arbitrary $\xi,\zeta\in c_{00}(X)$.

Inspired by \eqref{Eq.28.aug.26.dj.1}, consider  the sesquilinear form
\[(\xi,\zeta)\in c_{00}(X)\times c_{00}(X)\mapsto \int_0^t\langle \sigma_{h,s}(i\ad_h(a))\xi,\zeta\rangle ds\in \C.\]
It is immediate that this form is bounded  by $|t|\|\ad_h(a)\|$. So, there must be  $b\in \cB(\ell_2(X))$ such that
\[\langle b\xi,\zeta\rangle=\int_0^t\langle \sigma_{h,s}(i\ad_h(a))\xi,\zeta\rangle ds\]
for all $\xi,\zeta\in c_{00}(X)$. In particular, $\|b\|\leq|t|\|\ad_h(a)\|$. By \eqref{Eq.28.aug.26.dj.1},  \[b=\sigma_{h,t}(a)-a\] and the result follows.
\end{proof}

\begin{lemma}\label{Lemma.BandFromMoments}
Let $X$ be a set,  $h\colon X\to\R$ be a map, and  $a\in\cB(\ell_2(X))$ be such that $\ad_h^k(a)$ is bounded for all $k\in\N$. Assume      there are $A,B>0$ such that
\[\|\ad_h^k(a)\|\leq AB^kk!\ \text{ for all }\ k\geq0.\]
Then $a$ is analytic on  a strip for $\sigma_h$. Moreover,  this is the case for   any strip of width smaller than $ 1/B$.
\end{lemma}

\begin{proof}
Fix $0<\delta<1/B$. For   integers $j,k\geq 0$, define $b_{j,k}\colon [-\delta,\delta]\to \cB(\ell_2(X))$   by letting
\[b_{j,k}(s)=\frac{(-s)^k}{k!}\ad_h^{j+k}(a)\]
and define $b_{j}\colon [-\delta,\delta]\to \cB(\ell_2(X))$  by
\[b_j(s)=\sum_{k=0}^\infty b_{j,k}(s).\]
Notice that the series above converges  absolutely for integers $j\geq 0 $. Precisely, we have that
\[\|b_{j,k}(s)\|\leq A\frac{(j+k)!}{k!}B^{j+k}|s|^k=A B^j \frac{(j+k)!}{k!} B^k|s|^k.\]
As $B|s|<1$,
\[\sum_{k=0}^\infty B^k|s|^k=\frac{1}{1-B|s|}\] and, differentiating $j$ times both sides above with respect to $B$, one can conclude that

\begin{equation}\label{Eq.bj.Bound}\|b_j(s)\|\leq
A\sum_{k=0}^\infty\frac{(j+k)!}{k!}B^{j+k}|s|^k=\frac{AB^jj!}{(1-B|s|)^{j+1}}<\infty.
\end{equation}
In particular, $b_j$  is well defined. Moreover,   this shows that the sum defining $b_j$ converges uniformly on $[-\delta,\delta]$. Therefore, as each   $b_{j,k}$ is continuous, so is $b_j$.

We now define $F$ on the strip  $\{z\in \C\mid |\mathrm{Im}(z)|\leq \delta\}$ by letting, for each $z=t+is$ with $|s|\leq\delta$,
\[F(z)=\sigma_{h,t}(b_0(s)).\]
Notice that if $s=0$, then $b_0(0)=a$ and hence    \[F(t)=\sigma_{h,t}(a)\ \text{   for all }\ t\in\R,\] i.e., $F$ extends $t\in \R\mapsto\sigma_{h,t}(a)\in \cB(\ell_2(X))$. We are left to show that $F$ is continuous in its domain and analytic in the interior of its domain.

We start showing that $F$ is continuous. First notice that, for all $x,y\in X$, we have
\begin{equation}\label{Eq.bj.entries}\langle b_j(s)\delta_y,\delta_x\rangle=(h(x)-h(y))^je^{-s(h(x)-h(y))}\langle a\delta_y,\delta_x\rangle.\end{equation}
In particular, $\ad_h(b_0(s))=b_1(s)$ and,   by \eqref{Eq.bj.Bound}, we have
\[M=\sup_{|s|\leq\delta}\|\ad_h(b_0(s))\|\leq \frac{AB}{(1-B\delta)^{2}}<\infty.\]
Let now     $z=t+is$ and $z'=t'+is'$ be in the domain of $F$. Then,  applying Lemma \ref{Lemma.Duhamel},
\begin{align*}\|F(z)-F(z')\|&\leq \|b_0(s)-b_0(s')\|+\|\sigma_{h,t-t'}(b_0(s'))-b_0(s')\|\\
&\leq \|b_0(s)-b_0(s')\|+M|t-t'| .
\end{align*}
As $b_0$ is continuous, this shows that so is $F$.

Let us show $F$ is holomorphic in $\{z\in \C\mid |\mathrm{Im}(z)|<\delta\}$. Fix $z_0=t_0+is_0$ with $|s_0|<\delta$ and let us show that $F$ can be written as a power series around $z_0$. By \eqref{Eq.bj.Bound}, the series
\[G(z)=\sum_{j=0}^\infty \frac{i^j(z-z_0)^j}{j!}\,\sigma_{h,t_0}\big(b_j(s_0)\big)\]
converges absolutely if $|z-z_0|<1/B-|s_0|$. So, $G$ is holomorphic in this region. Moreover, it follows from  \eqref{Eq.bj.entries} that the matrix entries of $G(z)$ are
\[\sum_{j=0}^\infty\frac{(i(z-z_0)(h(x)-h(y)))^j}{j!}\,e^{it_0(h(x)-h(y))}e^{-s_0(h(x)-h(y))}\langle a\delta_y,\delta_x\rangle.\]
Since the latter equals $e^{iz(h(x)-h(y))}\langle a\delta_y,\delta_x\rangle$ and, by the formula of $F$, this equals $\langle F(z)\delta_y,\delta_x\rangle$, we conclude that
\[F(z)=G(z)\ \text{ for all }\ z\in \left\{z'\in \C \mid |\mathrm{Im}(z')|<\delta \ \text{ and }\ |z'-z_0|<\frac{1}{B}-|s_0|\right\}.\] This shows that  $F$ is holomorphic at $z_0$.
\end{proof}

\begin{lemma}\label{Lemma.IterSlicing}
Let $X$ be a connected finite  graph endowed with the shortest path metric,    $h\colon X\to\R$ be a map,   and $a\in\cB(\ell_2(X))$. Then, for all $k\geq1$,
\[\|\ad_h^k(a)\|\leq 2\,(3\Lip(h))^k\left(\|a\|+\sum_{m\geq2}m^k\,\eps_a(m)\right).\]
\end{lemma}

\begin{proof}
To simplify notation, let $L=\Lip(h)$. If $L=0$, the result is immediate, so assume $L>0$. For each $j\in\Z$, write \[E_j=h^{-1}([jL,(j+1)L)).\]    Let \[g=\sum_{j\in \Z} jL\chi_{E_j}\ \text{  and }\  b=h-g.\] So,   $0\leq b<L$. Notice that
\begin{equation}\label{Eq.BinomialSlicing}
\ad_h^k=(\ad_g+\ad_b)^k=\sum_{j=0}^k\binom kj \ad_g^j\ad_b^{k-j}
\end{equation}
(see \eqref{Eq.defi.adk.inductive}).
We will now estimate the norms of  $\ad_b^{k-j}(a)$ and of $\ad_g^j(a_j)$ where $a_j=\ad_b^{k-j}(a)$.

As $b$ is bounded, we can identify $b$ with a bounded operator in the diagonal $\ell_\infty(X)$ of $\cB(\ell_2(X))$. In particular,    $\ad_b(c)=[b,c]$ for all $c\in \cB(\ell_2(X))$. Moreover, as $b$ is bounded by $L$, we have  \begin{equation}\label{Eq.BinomialSlicing.bounds.1}\|\ad_b(c)\|\leq 2L\|c\|\ \text{  for all }\ c\in\cB(\ell_2(X)).\end{equation}
For $A,B\subseteq X$, we have \[ \chi_A\ad_b(c)\chi_B=b\chi_A\cdot \chi_Ac\chi_B-\chi_Ac\chi_B\cdot b\chi_B\] which implies that   \begin{equation}\label{Eq.BinomialSlicing.bounds.2}\eps_{\ad_b(c)}\leq2L \eps_c \ \text{  for all }\ c\in\cB(\ell_2(X)).\end{equation}
Hence, for $k\geq j$, letting $a_j=\ad_b^{k-j}(a)$, \eqref{Eq.BinomialSlicing.bounds.1} and \eqref{Eq.BinomialSlicing.bounds.2} give
\begin{equation}\label{Eq.AdbBounds}
\|a_j\|\leq(2L)^{k-j}\|a\|\ \text{ and }\ \eps_{a_j}\leq(2L)^{k-j}\eps_a.
\end{equation}

Given $c\in \cB(\ell_2(X))$ and   $m\in\Z$, let
\[T_m(c)=\sum_{j\in\Z}\chi_{E_{j+m}}c\,\chi_{E_{j}}.\]
As $X$ is finite, it is clear that
\[c=\sum_{m\in\Z} T_m(c).\]
Hence, since $g$ equals $mL$ in each $E_m$, we have \begin{equation}\label{Eq.AdbBounds.223.0}\ad^j_g(c)=\sum_{m\in\Z}(  mL)^jT_m(c).\end{equation}
We estimate the quantity above depending on $m$. Firstly, for an arbitrary $m\in\Z$,  the summands of $T_m(c)$ have pairwise orthogonal domains and ranges, so \begin{equation}\label{Eq.AdbBounds.223}\|T_m(c)\|=\max_{j\in\Z}\|\chi_{E_{j+m}}c\chi_{E_{j}}\|\leq\|c\|.\end{equation}
Now, if  $|m|\geq2$,  then $|h(x)-h(y)|>(|m|-1)L$ for all  $x\in E_{j+m}$ and $y\in E_{j}$. Since $h$ is $L$-Lipschitz and distances in $X$ are integers,  this implies that $d(x,y)\geq|m|$ and so   $d(E_{j+m},E_{j})\geq|m|$. Therefore, the equality in \eqref{Eq.AdbBounds.223} gives \begin{equation}\label{Eq.AdbBounds.223.1}\|T_m(c)\|\leq\eps_c(|m|).\end{equation}

Equations  \eqref{Eq.AdbBounds.223.0}, \eqref{Eq.AdbBounds.223}, and \eqref{Eq.AdbBounds.223.1} together give
\[\|\ad_g^j(a_j)\|\leq 2L^j\Big(\|a_j\|+\sum_{m\geq2}m^j\eps_{a_j}(m)\Big).\]
Hence, by \eqref{Eq.AdbBounds} and since $j\leq k $, we have
\[\|\ad_g^j(a_j)\|\leq 2\cdot 2^{k-j}L^k\Big(\|a\|+\sum_{m\geq2}m^k\eps_{a}(m)\Big).\]
As $\sum_j\binom kj2^{k-j}=3^k$, the result follows from \eqref{Eq.BinomialSlicing}.
\end{proof}

\begin{proof}[Proof of Theorem \ref{Thm.Led.In.Band}]
As $X$ is a u.l.f.\ metric space which is a coarse disjoint union of finite connected graphs, write  $X=\bigsqcup_nX_n$ where each $X_n$ is a finite connected graph.
Since $\AP_{\mathrm{strip}}(X)$ is closed, it is enough to show that $ \AP_{\mathrm{strip}}(X)$ contains all contractions in $\cB(\ell_2(X))$ whose quasi-locality modulus decays exponentially. Let $a\in \cB(\ell_2(X))$ be such a contraction and fix $c,C>0$ such that
\[\eps_a(r)\leq Ce^{-cr}\ \text{ for all }r>0.\]
 As   $a\in \cstql(X)$,   $a-\SOTh\sum_n\chi_{X_n}a\chi_{X_n}$ is compact and, in particular, it belongs to $ \AP_{\mathrm{strip}}(X)$. Therefore, we can assume without loss of generality that
\[a=\SOTh\sum_n\chi_{X_n}a\chi_{X_n}.\]

Fix a coarse $h\colon X\to\R$ and let us show that $a$ is analytic on a strip for $\sigma_h$. Let $L=\max\{1,\omega_h(1)\}$. Notice that, as $a$ is in $\prod_n\cB(\ell_2(X_n))$,  \[\ad_h^k(a)=\SOTh\sum_{n\in\N}  \ad_{h|_{X_n}}^k(\chi_{X_n}a\chi_{X_n}).\] Therefore, since
\[\eps_{\chi_{X_n}a\chi_{X_n}}\leq\eps_a\ \text{ for all }\ n\in\N,\]
Lemma \ref{Lemma.IterSlicing} applied to each block gives that
\begin{align}\label{Eq.lunch.break}\|\ad_h^k(a)\| &=\sup_n\|\ad^k_{h|_{X_n}}(\chi_{X_n}a\chi_{X_n})\|\\
&\leq 2(3L)^k\Big(1+C\sum_{m\geq2}m^ke^{-cm}\Big)\notag\end{align}
for all $k\in\N$. As for  $t\in [m-1,m]$ one has $m\leq t+1$ and $  e^{-cm}\leq e^{-ct}$, we have
 \[\sum_{m\geq2}m^ke^{-cm}\leq\int_1^\infty(t+1)^ke^{-ct}dt \leq e^{c}\frac{k!}{c^{k+1}}.\]
Together with \eqref{Eq.lunch.break}, this  gives
\begin{align*}\|\ad_h^k(a)\| & \leq 2(3L)^k\Big(1+Ce^{c}\,\frac{k!}{c^{k+1}}\Big)\\
 &\leq  2\Big(1+ \frac{Ce^{c}}{c}\Big)\left(\frac{3L}{\min\{1,c\}}\right)^k k! \notag\end{align*}
 for all $k\geq 0$.
 By Lemma \ref{Lemma.BandFromMoments}, $a$ is analytic on a strip for   $\sigma_h$.
\end{proof}

\begin{proof}
    [Proof of Theorem \ref{thmE}]
By Theorem  \ref{Thm.Band.In.Led}, $\mathrm{AP}_{\mathrm{strip}}(X)\subseteq \mathrm{QL}_{\exp}(X)$ and by Theorem \ref{Thm.Led.In.Band} $\mathrm{QL}_{\exp}(X)\subseteq \mathrm{AP}_{\mathrm{strip}}(X)$ for u.l.f.\ metric spaces which are coarse disjoint unions of connected  finite graphs. So, these algebras are equal for such $X$'s.
\end{proof}
 
\begin{proof}[Proof of Theorem \ref{thmC}]
    By Theorem \ref{thmD}, we have 
    \[\cstu(X)\subsetneq \mathrm{QL}_{\exp}(X)\subsetneq \cstql(X)\]
    and, by Theorem \ref{thmE}, $\mathrm{AP}_{\mathrm{strip}}(X)=\mathrm{QL}_{\exp}(X)$. Therefore,
    \[\cstu(X)\subsetneq \mathrm{AP}_{\mathrm{strip}}(X)\subsetneq \cstql(X)\]
    and we are done.
\end{proof}

\appendix

\section{A consequence of measure concentration} \label{Appendix}
This appendix contains the proofs of Lemmas  \ref{LemmaGoodSubspacesQuarterPower} and \ref{LemmaGoodSubspacesQuarterPower.LARGEDIM}. Earlier drafts of this manuscript had the detailed arguments for these lemmas as corollaries of the concentration of measure phenomena. This is however not necessary any more since \cite{LiZhangZhu2026} provided the precise result needed for us. Precisely:

\begin{lemma}\emph{(}\cite[Proposition 6.3]{LiZhangZhu2026}\emph{)}. There is a constant $C>0$ such that, for all $n\in\N$ and all $k\in\{1,\dots,n\}$,    there is a subspace $W\subseteq\C^n$ with $\dim_{\C}(W)=k$ such that
\begin{equation}\label{eq:framebound.complex}
\|\chi_Ap_W\|\leq C\sqrt{\frac kn+\frac{|A|}{n}\log\left(\frac{en}{|A|}\right)}
\end{equation}
for all nonempty $A\subseteq\{1,\dots,n\}$.\label{cor.lem:frame}
\end{lemma}

\begin{proof}
    [Proof of Lemma \ref{LemmaGoodSubspacesQuarterPower}]
 Let $k=\lceil n^{1/4}\rceil$. By Lemma \ref{cor.lem:frame}, there is a universal $C>0$ and  a subspace $W\subseteq \C^n$ of (complex) dimension $k$ such that \begin{equation}\label{eq:framebound.222}
\|\chi_Ap_W\|\leq C\sqrt{\frac kn+\frac{|A|}{n}\log\left(\frac{en}{|A|}\right)}
\end{equation}
for all nonempty $A\subseteq\{1,\dots,n\}$.

We only need to notice that $W$ has the required properties; this will follow from  some very simple estimates. For that, fix $\delta\in[1/k,1/2]$ and nonempty $A\subseteq\{1,\dots,n\}$ with $|A|\leq n\delta$. As $1/k\leq \delta\leq 1/2$ and $k\leq 2n^{1/4}$, we have that $\log(1/\delta)\geq\log2$ and $2\delta\geq1/n^{1/4}$. Hence,  
\[
\frac kn\leq\frac{2n^{1/4}}{n}=\frac{2}{n^{1/2}}\cdot\frac{1}{n^{1/4}}\leq\frac{4\delta}{n^{1/2}}\leq4\delta\leq\frac{4}{\log2}\,\delta\log\Big(\frac1\delta\Big).
\]
Moreover, as $|A|/n\leq\delta\leq1$ and as the map $t\mapsto t\log(e/t)$ is increasing on $(0,1]$,
\[
\frac{|A|}{n}\log\Big(\frac{en}{|A|}\Big)\leq\delta\log\Big(\frac e\delta\Big)=\delta+\delta\log\Big(\frac1\delta\Big)\leq\Big(1+\frac1{\log2}\Big)\delta\log\Big(\frac1\delta\Big).
\]
Since $\frac4{\log2}+1+\frac1{\log2}\leq9$, \eqref{eq:framebound.222} implies that  \[\|\chi_Ap_W\|\leq C\sqrt{9\,\delta\log\left(\frac{1}{\delta}\right)}=3C\sqrt{\delta\log\left(\frac{1}{\delta}\right)}\]
and we are done.
\end{proof}

\begin{proof}[Proof of Lemma \ref{LemmaGoodSubspacesQuarterPower.LARGEDIM}]
Let $W\subseteq \C^n$ be the vector space of dimension  $k=\lfloor\tfrac{n}{(\log{n})^\alpha}\rfloor$ given  by    Lemma \ref{cor.lem:frame}. As $k/n\leq 1/(\log(n))^\alpha$, the result follows.
\end{proof}

\begin{AIusage}
This paper started as a natural continuation of \cite{BragaExel2023,BragaBussExel2026Jussieu}   in  2024; these papers led the author to initiate the program of  characterizing the  uniform Roe and the quasi-local algebra dynamically as well as to  find natural dynamical algebras strictly in between them for the cases in which these algebras are known not to coincide. After the release of Claude Fable 5.0, the paper developed through     extended interactions with this model. This includes interactions on  the proofs of all the main results in the paper.  The organization and exposition of the paper are the author's responsibility only. Moreover, all work done in collaboration with Claude Fable 5.0 was independently checked and reworked by the author,   who also takes full responsibility for the correctness, originality, and integrity of the work. 
\end{AIusage}

\begin{leanformalization}
Formalizations of Theorems A--E and their supporting results
were developed using GPT-6 Astra. The resulting proofs were
checked by the Lean 4 kernel (version 4.33.1), using Mathlib
v4.33.1 and the additional developments included in the repository.
The source files, a correspondence between the manuscript
and the formal statements, and instructions for reproducing
the verification are available at
\begin{center}
\url{https://github.com/demendoncabraga/dynamical-cstar-algebras}.
\end{center}
\end{leanformalization}

\begin{acknowledgments}
    The author wishes to thank Alcides Buss and Ruy Exel for conversations about this project. The author is specially thankful to Alcides Buss for asking  during a visit to IMPA in January of 2024 if the continuity points of a diagonal flow given by an arbitrary (not necessarily coarse) map in a given uniform Roe algebra could be seen as the uniform Roe algebra of a different coarse structure on the space (see Theorem  \ref{Thm.qla.h.Points.Cont.Substructure}).  The author is deeply grateful to Cynthia Bortolotto and João Pedro Ramos for sharing their Lean formalization repository, without which the Lean formalization of  this work would not have been possible.
\end{acknowledgments}

 \bibliographystyle{amsalpha}
 \bibliography{bibliography}

\end{document}